\documentclass[preprint,3p]{elsarticle}

\usepackage{graphics}
\usepackage{graphicx}
\usepackage{epsfig}
\usepackage{amsmath}
\usepackage{graphics}
\usepackage{graphicx}
\usepackage{epstopdf}
\usepackage{multirow}
\usepackage{booktabs}
\usepackage{subfigure}
\usepackage{ifpdf}
\usepackage{color}
\usepackage{cite}
\usepackage{amssymb}
\usepackage{amsthm}
\usepackage{amsmath}
\usepackage{matlab-prettifier}
\usepackage[numbers]{natbib}
\usepackage{lineno}
\usepackage{xcolor}
\usepackage[colorinlistoftodos]{todonotes}
\usepackage{ulem}
\newtheorem{proposition}{Proposition}[section]
\newtheorem{definition}{Definition}[section]
\newtheorem{theorem}{Theorem}[section]
\newtheorem{corollary}{Corollary}[section]
\newtheorem{lemma}{Lemma}[section]
\newtheorem{example}{Example}[section]

\usepackage{lineno}

\numberwithin{equation}{section}

\newcommand{\Q}{\mathbb{Q}}
\newcommand{\R}{\mathbb{R}}
\newcommand{\C}{\mathbb{C}}

\newcommand{\bpp}{\big|\big|}
\newcommand{\Bpp}{\Big|\Big|}

\DeclareMathOperator{\tr}{tr}
\DeclareMathOperator{\bcirc}{bcirc}
\DeclareMathOperator{\unfold}{unfold}
\DeclareMathOperator{\fold}{fold}
\DeclareMathOperator{\diag}{diag}
\DeclareMathOperator{\spec}{spec}

\usepackage{ mathrsfs }

\usepackage{algorithm}
\usepackage{algpseudocode}
\journal{}

\begin{document}

\begin{frontmatter}

\title{Geometric mean for T-positive definite tensors \\and associated Riemannian geometry}

\author{Jeong-Hoon Ju$^{1}$}
\ead{jjh793012@naver.com}

\author{Taehyeong Kim$^{2}$}
\ead{thkim0519@knu.ac.kr}

\author{Yeongrak Kim$^{1, 3}$}
\ead{yeongrak.kim@pusan.ac.kr}

\author{Hayoung Choi$^{2,4^*}$}
\ead{hayoung.choi@knu.ac.kr}

\address{
$^1$Department of Mathematics, Pusan National University, 
2 Busandaehak-ro 63 beon-gil, Geumjeong-gu, 46241 Busan, Republic of Korea}
\address{
$^2$Nonlinear Dynamics and Mathematical Application Center, Kyungpook National University,80, Daehak-ro, Buk-gu, 41566 Daegu, Republic of Korea}
\address{
$^3$Institute of Mathematical Science, Pusan National University, 
2 Busandaehak-ro 63 beon-gil, Geumjeong-gu, 46241 Busan, Republic of Korea}
\address{
$^4$Department of Mathematics, Kyungpook National University, 
80, Daehak-ro, Buk-gu, 41566 Daegu, Republic of Korea}
\cortext[cor1]{corresponding author. Hayoung Choi (hayoung.choi@knu.ac.kr)}

\begin{abstract}
In this paper, we generalize the geometric mean of two positive definite matrices to that of third-order tensors using the notion of T-product. Specifically, we define the geometric mean of two T-positive definite tensors and verify several properties that ``mean"  should satisfy including the idempotence and the commutative property, and so on. Moreover, it is shown that the geometric mean is a unique T-positive definite solution of an algebraic Riccati tensor equation and can be expressed as solutions of algebraic Riccati matrix equations. In addition, we investigate the Riemannian manifold associated with the geometric mean for T-positive definite tensors, considering it as a totally geodesic embedded submanifold of the Riemannian manifold associated with the case of matrices. It is particularly shown that the geometric mean of two T-positive definite tensors is the midpoint of a unique geodesic joining the tensors, and the manifold is a Cartan-Hadamard-Riemannian manifold.
\end{abstract}


\end{frontmatter}

\section{Introduction}

The geometric mean of two $N \times N$ Hermitian (or symmetric) positive definite matrices $A,B$ is defined as
\begin{equation}\label{eqMatrixGeometricMean}
	A \# B:=A^{\frac{1}{2}}(A^{-\frac{1}{2}}B A^{-\frac{1}{2}})^{\frac{1}{2}}A^{\frac{1}{2}}.
\end{equation}
The geometric mean was introduced by Pusz and Woronowicz \citep{pusz1975functional}, and by  Ando \citep{ando1978topics} for positive operators. Then Kubo and Ando \citep{kubo1980means} proved some results on L\"owner inequalities related to the geometric mean. Moreover, Lawson and Lim \citep{lawson2001geometric} showed that geometric averages satisfy properties that ``means'' generally should have, such as the idempotence and the commutative property.

In addition, the Riemannian geometry associated with the geometric mean is studied by Moakher \citep{moakher2005differential}, and by Bhatia and Holbrook \citep{bhatia2006riemannian}. The Riemannian metric with respect to $P\in {\mathcal P}_N$ is defined on the convex cone ${\mathcal P}_N$ of $N \times N$ Hermitian positive definite matrices as the trace metric $\tr (P^{-1}XP^{-1}Y),$ where $X,Y$ are in the Euclidean space ${\mathcal H}_N$ of $N \times N$ Hermitian matrices. 
The Riemannian distance between $A,B \in{\mathcal P}_N$ with respect to the above metric is given by $\delta(A,B) = \| \log (A^{-1/2}BA^{-1/2}) \|_F$, and the unique (up to parametrization) geodesic joining $A$ and $B$ is given as the curve of weighted geometric means,
\begin{equation}\label{eqw-g-mean}
t \in [0,1] \ \, \longmapsto \ \, A \#_{t} B := A^{1/2} (A^{-1/2} B A^{-1/2})^{t} A^{1/2}
\end{equation}
so that the geometric mean $A \# B$ is the midpoint of the geodesic. Moreover, this Riemannian manifold ${\mathcal P}_N$ is a Cartan-Hadamard manifold, that is, a simply connected complete manifold with nonpositive (sectional) curvature.

The weighted geometric mean is intimately connected to the concept of matrix interpolation, providing a smooth transition between the matrices $A$ and $B$. And it, including geometric mean, also has roles in several matrix equations revealing its utility in various applications including statistic mechanics \citep{dinh2022some, jung2009solution, lee2022nonlinear, lee2020nonlinear} and quantum theory \citep{cree2020fidelity}. The associated Riemannian geometry also has various ongoing researches as its applications, including data clustering \citep{piao2020kernel, zheng2017clustering} and visual recognition \citep{cherian2016positive, dong2017deep}.

There are two basic ways to generalize the notion of geometric mean; one is for three or more positive matrices, and the other is for tensors. There are several attempts on the study of geometric mean for several positive matrices, such as Ando-Li-Mathias mean \citep{ando2004geometric} or Karcher mean \citep{bhatia2006riemannian, moakher2005differential},  offering a broader spectrum of applications in multivariate statistical analysis and information geometry. However, the generalization for tensors has not been understood very much.

Tensors provide a framework for generalizing algebraic operations over vectors and matrices, enabling the compact representation and manipulation of high-dimensional data \citep{qi2017tensor}. This mathematical structure is particularly valuable in fields such as matrix analysis and numerical linear algebra, as it enables and eases complicated operations in machine learning and data analysis \citep{kolda2009tensor}.

As a linear map can be represented by a matrix, a multilinear map can be represented by a tensor, that is, a tensor is a multilinear map. A big difference between two notions appears when we define a multiplication. The matrix multiplication naturally arises  since it represents the composition of the corresponding linear maps. On the other hand, the composition of multilinear maps does not naturally appear in general, so it is hard to generalize the matrix multiplication to tensor-tensor multiplication. This is one of  reasons why there has been less study on geometric mean for tensors.

Hence, before defining a geometric mean of two tensors, we need a convention how to multiply two given tensors. In this paper we use the tool known as the \emph{T-product} of tensors that Kilmer, Martin, and Perrone suggested \citep{kilmer2011factorization, kilmer2008third}. Although it does not represent the composition of multilinear maps,  Braman \citep{braman2010third} showed that it represents the composition of linear maps on specific finitely generated free modules. 
After that, the research on the T-product was carried out in many aspects: tensor function theory \citep{lund2020tensor, miao2020generalized}, T-Jordan canonical form \citep{miao2021t}, tensor inequalities \citep{cao2023some}, perturbation theory \citep{cao2022perturbation, cong2022characterizations, mo2021perturbation}, and applications to imaging data \citep{chen2023perturbations, tarzanagh2018fast}. In particular, Zheng, Huang and Wang defined T-positive (semi)definite tensor and introduced T-(semi)definite programming \citep{zheng2021t}. The principal aim of this study is to extend the notion of geometric mean from matrices to tensors while carefully examining various properties.

The structure of this paper is as follows. In Section \ref{Sect:Preliminaries}, we review the notion of T-product and T-positive definite tensor and basic properties of them. In Section \ref{Sect:Geometric Mean of T-positive Definite Tensors}, we define the geometric mean of T-positive definite tensors and prove some properties which ``mean'' has to satisfy in the sense of Lawson-Lim \citep{lawson2001geometric}. In addition, we define T-L\"owner order for Hermitian tensors and prove some inequalities. In Section \ref{Sect:Riemannian Geometry}, we introduce a Riemannian metric on the convex open cone of T-positive definite tensors, and interpret the geometric mean in terms of this Riemannian metric. In particular, we prove that the geometric mean of two T-positive definite tensors $\mathcal{A},\mathcal{B}$ is the midpoint of the geodesic from $\mathcal{A}$ to $\mathcal{B}$ in this Riemannian manifold, and this Riemannian manifold is complete and has nonpositive curvature.

\bigskip
\section{Preliminaries}\label{Sect:Preliminaries}

In this section, we review some notations, definitions, and their basic properties from \citep{kilmer2011factorization, miao2021t, zheng2021t}. Although all the results in the references were established over $\R$, it is easy to see that the same results hold over $\C$.

A \emph{$k$th-order tensor} is a hyperarray $\mathcal{A}=[a_{i_1i_2\cdots i_k}] \in \C^{n_1 \times n_2 \times \cdots \times n_k}$, where $a_{i_1i_2\cdots i_k}$ denotes the $(i_1,i_2,...,i_k)$-th entry of $\mathcal{A}$. In this paper, we only consider third-order tensors $\mathcal{A} \in \C^{m \times n \times p}$.
A third-order tensor $\mathcal{A} \in \C^{m \times n \times p}$ can be naturally considered as a stack of $p$ frontal slices $A^{(1)}, \ldots, A^{(p)} \in \C^{m \times n}$. We denote the third-order tensor $\mathcal{A}$ as
\begin{equation*}
	\mathcal{A}=\left[A^{(1)}\Bpp A^{(2)}\Bpp \cdots \Bpp A^{(p)}\right].
\end{equation*}

\bigskip
\subsection{T-product}

The first main ingredient we need is a method to multiply two given tensors. \emph{T-product} is one of the generalized notions for multiplying two tensors, and it is defined via the operators, namely,  \emph{block circulant matricizing} and \emph{unfolding}.

\begin{definition}
	Let $\mathcal{A}=\left[A^{(1)}\bpp A^{(2)} \bpp \cdots \bpp A^{(p)}\right] \in \C^{m \times n \times p}$. We define the block circulant matricizing operator $\bcirc$ and the unfolding operator $\unfold$ as the concatenated matrices
	\begin{equation*}
		\bcirc(\mathcal{A})=\begin{bmatrix}
			A^{(1)} & A^{(p)} & A^{(p-1)} & \cdots & A^{(2)}\\
			A^{(2)} & A^{(1)} & A^{(p)} & \cdots & A^{(3)}\\
			A^{(3)} & A^{(2)} & A^{(1)} & \cdots & A^{(4)}\\
			\vdots & \vdots & \vdots & \ddots & \vdots\\
			A^{(p)} & A^{(p-1)} & A^{(p-2)} & \cdots & A^{(1)}
		\end{bmatrix}~~\text{and}~~\unfold(\mathcal{A})=\begin{bmatrix}
			A^{(1)}\\
			A^{(2)}\\
			\vdots\\
			A^{(p)}
		\end{bmatrix}.
	\end{equation*}
	We also define operators $\bcirc^{-1}$ and $\fold$  as the inverse of $\bcirc$ and $\unfold$, respectively, i.e.,
	\begin{equation*}
		\bcirc^{-1}(\bcirc(\mathcal{A}))=\mathcal{A}~~\text{and}~~\fold(\unfold(\mathcal{A}))=\mathcal{A}.
	\end{equation*}
\end{definition}

\begin{definition}[\citep{kilmer2011factorization}]
	Let $\mathcal{A}=\left[A^{(1)}\bpp A^{(2)} \bpp \cdots \bpp A^{(p)}\right] \in \C^{m \times n \times p}$ and $\mathcal{B}=\left[B^{(1)} \bpp B^{(2)}\bpp \cdots \bpp B^{(p)}\right] \in \C^{n \times s \times p}$. Then the T-product $\mathcal{A} * \mathcal{B}$ of $\mathcal{A}$ and $\mathcal{B}$ is defined as an $m \times s \times p$ tensor
	\begin{equation*}
    		\mathcal{A} * \mathcal{B}=\fold(\bcirc(\mathcal{A})\cdot \unfold(\mathcal{B})),
	\end{equation*}
	where $\cdot$ denotes matrix multiplication.
\end{definition}

For example, when $\mathcal{A}=\left[A^{(1)} \bpp A^{(2)} \bpp A^{(3)}\right] \in \C^{m \times n \times 3}$ and $\mathcal{B}=\left[B^{(1)} \bpp B^{(2)} \bpp B^{(3)}\right] \in \C^{n \times s \times 3}$, 
\begin{align*}
		&\mathcal{A}*\mathcal{B}=\fold\left(\begin{bmatrix}
			A^{(1)} & A^{(3)} & A^{(2)}\\
			A^{(2)} & A^{(1)} & A^{(3)}\\
			A^{(3)} & A^{(2)} & A^{(1)}
		\end{bmatrix}\begin{bmatrix}
			B^{(1)}\\
			B^{(2)}\\
			B^{(3)}
		\end{bmatrix} \right)
		=\fold\left( \begin{bmatrix}
			A^{(1)}B^{(1)}+A^{(3)}B^{(2)}+A^{(2)}B^{(3)}\\
			A^{(2)}B^{(1)}+A^{(1)}B^{(2)}+A^{(3)}B^{(3)}\\
			A^{(3)}B^{(1)}+A^{(2)}B^{(2)}+A^{(1)}B^{(3)}
		\end{bmatrix} \right)\\
		&=\left[A^{(1)}B^{(1)}+A^{(3)}B^{(2)}+A^{(2)}B^{(3)}\Bpp
			A^{(2)}B^{(1)}+A^{(1)}B^{(2)}+A^{(3)}B^{(3)}\Bpp
			A^{(3)}B^{(1)}+A^{(2)}B^{(2)}+A^{(1)}B^{(3)}\right].
\end{align*}

The following properties are easy to verify and useful.

\begin{lemma}[\citep{kilmer2011factorization, miao2021t}]\label{Lemma:BasicPropertiesTproduct}
	Let $\mathcal{A} \in \C^{m \times n \times p}, \mathcal{B} \in \C^{n \times s \times p}$ and $\mathcal{C} \in \C^{s \times r \times p}$. Then
	\begin{itemize}
		\item [(\romannumeral1)] $\bcirc(\mathcal{A}*\mathcal{B})=\bcirc(\mathcal{A})\cdot \bcirc(\mathcal{B})$.
		\item [(\romannumeral2)] $(\mathcal{A} * \mathcal{B})* \mathcal{C}=\mathcal{A} *( \mathcal{B}* \mathcal{C})$.
	\end{itemize}
\end{lemma}

Lemma \ref{Lemma:BasicPropertiesTproduct}.(\romannumeral1) makes sense since the product of two block circulant matrices is again a block circulant matrix. Note that T-product of $m \times n \times p$ and $n \times s \times p$ tensors gives a $m \times s \times p$ tensor. As a square matrix does not change the size of matrices by matrix multiplication, a tensor whose frontal slices are square matrices does not change the size of tensors by T-product; we call such a tensor a \emph{frontal square tensor}.

\begin{definition}[\citep{kilmer2011factorization}]
	The $n \times n \times p$ \emph{identity tensor} $\mathcal{I}_{n,p}$ is the tensor whose first frontal slice is the $n \times n$ identity matrix $I_n$, and whose other frontal slices are all $n \times n$ zero matrices $O_n$, i.e.,  $\mathcal{I}_{n,p}=\left[I_n \bpp O_n \bpp \cdots \bpp O_n\right]$ so that $\bcirc(\mathcal{I}_{n,p})=I_{np}$.
\end{definition}

\begin{definition}[\citep{kilmer2011factorization}]
	A frontal square tensor $\mathcal{A} \in \C^{n \times n \times p}$ is said to be \emph{invertible} (or \emph{nonsingular}) if it has an inverse tensor $\mathcal{X} \in \C^{n \times n \times p}$ such that
	\begin{equation*}
		\mathcal{A}*\mathcal{X}=\mathcal{X}*\mathcal{A}=\mathcal{I}_{n,p},
	\end{equation*}
	and denote the inverse of $\mathcal{A}$ by $\mathcal{A}^{-1}$. If $\mathcal{A}$ has no inverse, then we say that $\mathcal{A}$ is \emph{singular}.
\end{definition}

Note that the inverse of a block circulant matrix is also a block circulant \citep{trapp1973inverses}. The invertibility of a third-order frontal square tensor $\mathcal{A}$ is equivalent to the invertibility of its block circulant matricization $\bcirc(\mathcal{A})$ by Lemma \ref{Lemma:BasicPropertiesTproduct}.(\romannumeral1).

\begin{lemma}\label{LemmaInvertibility}
	For a frontal square tensor $\mathcal{A} \in \C^{n \times n \times p}$, $\mathcal{A}$ is invertible if and only if $\bcirc(\mathcal{A})$ is invertible.
\end{lemma}

\medskip

Recall that each circulant matrix can be diagonalized with the normalized discrete Fourier transform (DFT) matrix \citep{golub2013matrix}. These phenomena also arise for block circular matrices. For a matrix $A$, let $A^H$ denote the conjugate transpose of $A$.
We denote the block diagonal matrix with diagonal blocks $A_1,...,A_p$ by $\diag(A_1,...,A_p)$.

\begin{lemma}[\citep{kilmer2011factorization}]\label{LemmaFourierBlockDiag}
	Let $\mathcal{A}=\left[A^{(1)} \bpp A^{(2)} \bpp \cdots \bpp A^{(p)}\right] \in \C^{m \times n \times p}$. Then there exist $A_1,...,A_p \in \C^{m \times n}$ such that
	\begin{equation}\label{eqFourierBlockDiagMN}
		\bcirc(\mathcal{A})=(\mathbf{F}_p^H \otimes I_m)\cdot \operatorname{diag}(A_1,...,A_p)\cdot (\mathbf{F}_p \otimes I_n),
	\end{equation}
	where
	\begin{equation*}
		\mathbf{F}_p=\frac{1}{\sqrt{p}}\begin{bmatrix}
			1 & 1 & 1 & \cdots & 1\\
			1 & \omega & \omega^2 & \cdots & \omega^{p-1}\\
			\vdots & \vdots & \vdots & \ddots & \vdots\\
			1 & \omega^{p-1} & \omega^{2(p-1)} & \cdots & \omega^{(p-1)(p-1)}
		\end{bmatrix}~~\text{for}~~\omega=e^{\frac{2 \pi i}{p}}.
	\end{equation*}
\end{lemma}

In addition, when $n=m$, the diagonal blocks $A_1 ,\ldots , A_p$ are of the form as
\begin{equation*}
    A_i = \sum\limits_{k=1}^{p} \omega^{(i-1)(k-1)}A^{(k)}.
\end{equation*}

\medskip
\subsection{T-positive definite tensors}

Before introducing the notion of T-positive definite tensor, we review \emph{T-Hermitian} tensor and \emph{Frobenius inner product} for tensors.

\begin{definition}[\citep{kilmer2011factorization}]
	For $\mathcal{A}=\left[A^{(1)}\bpp A^{(2)} \bpp \cdots  \bpp A^{(p)}\right] \in \C^{m \times n \times p}$, the \emph{T-conjugate transpose} $\mathcal{A}^H$ is defined as conjugate transposing each of the frontal slices and then reversing the order of transposed frontal slices $2$ through $p$:
$$
\mathcal{A}^H=\left[\left(A^{(1)}\right)^{H}\Bpp \left(A^{(p)}\right)^{H}\Bpp\cdots\Bpp\left(A^{(2)}\right)^{H}\right].
$$ 
\end{definition}

The process of reversing the order of transposed frontal slices $2$ through $p$ looks unnatural, however, it is natural in the perspective of block circulant matricizing as follows.

\begin{lemma}[\citep{kilmer2011factorization}]\label{LemmaHermitian}
	 $\bcirc(\mathcal{A}^H)=\bcirc(\mathcal{A})^H$ for all $\mathcal{A} \in \C^{m \times n \times p}$.
\end{lemma}
	
\begin{definition}[\citep{kilmer2011factorization}]
	For a frontal square tensor $\mathcal{A} \in \C^{n \times n \times p}$, we say $\mathcal{A}$ is \emph{T-Hermitian} if $\mathcal{A}^H=\mathcal{A}$, or equivalently $\bcirc(\mathcal{A})^H=\bcirc(\mathcal{A})$. 
\end{definition}
The set of all $n \times n \times p$ complex T-Hermitian tensors is denoted by $\mathbb{H}^{n \times n \times p}$.

\begin{definition}[\citep{kilmer2011factorization}]
	Let $\mathcal{A}=[a_{ijk}]=\left[A^{(1)}\bpp A^{(2)}\bpp \cdots\bpp A^{(p)}\right],~ \mathcal{B}=\left[b_{ijk}]=[B^{(1)}\bpp B^{(2)}\bpp \cdots\bpp B^{(p)}\right] \in \C^{m \times n \times p}$.
	We define the \emph{Frobenius inner product} $\langle \cdot , \cdot \rangle_{\mathcal{F}}$  on $\C^{m \times n \times p}$  by
	\begin{equation*}
		\langle \mathcal{A},\mathcal{B} \rangle_{\mathcal{F}}:=\sum_{i=1}^{m}\sum_{j=1}^n\sum_{k=1}^p \overline{a_{ijk}}b_{ijk}
  =\sum_{k=1}^p \langle A^{(k)}, B^{(k)} \rangle_{F},
	\end{equation*}
	where the $\langle \cdot, \cdot \rangle_{F}$ at the right-hand side denotes the usual Frobenius inner product on $\C^{m \times n}$. Then the Frobenius norm associated with the inner product is as follows: 
	\begin{equation*}
		||\mathcal{A}||_{\mathcal{F}}=\sqrt{\langle \mathcal{A}, \mathcal{A} \rangle_{\mathcal{F}}}=\sqrt{\sum_{k=1}^p \langle A^{(k)},A^{(k)} \rangle_F}.
	\end{equation*}
	We simply denote the Frobenius inner product $\langle \cdot , \cdot \rangle_{\mathcal{F}}$ and norm $||\cdot||_{\mathcal{F}}$ as $\langle \cdot,\cdot \rangle$ and $||\cdot||$, respectively, unless there is any confusion.
\end{definition}

\begin{definition}[\citep{zheng2021t}]\label{DefinitionPositiveDefinite}
	Let $\mathcal{A} \in \mathbb{H}^{n \times n \times p}$. We say $\mathcal{A}$ is a \emph{T-positive (semi-)definite tensor} if 
	\begin{equation*}
		\langle \mathcal{X},\mathcal{A}*\mathcal{X} \rangle > 0 ~~(\geq 0)
	\end{equation*}
	for any $\mathcal{X}\in \C^{n \times 1 \times p} \setminus \{\mathcal{O}\}$, where $\mathcal{O}$ denotes the zero tensor whose all entries are zero. We denote the set of all $n \times n \times p$ T-positive (semi-)definite tensors as $\mathbb{H}_{++}^{n \times n \times p}$ ( $\mathbb{H}_{+}^{n \times n \times p}$).
\end{definition}

Recall that the positive (semi-)definiteness of a Hermitian matrix $A \in \C^{n \times n}$ is defined by the inequality  
	\begin{equation*}
		\langle x, Ax \rangle>0~~( \geq 0)
	\end{equation*}
	for every nonzero $x \in \C^n$, where $\langle \cdot, \cdot \rangle$ denotes the Frobenius inner product on $\C^n$. By considering $\mathcal{X} \in \C^{n \times 1 \times p}$ as a matrix in $\C^{n \times p}$, we may regard Definition \ref{DefinitionPositiveDefinite} as a generalization from the matrix case. Furthermore, the convention to identify a matrix in $\C^{n \times p}$ to a tensor in $\C^{n \times 1 \times p}$ has an algebraic meaning when tensors $\C^{n \times n \times p}$ are considered as linear operators on a space of matrices \citep{braman2010third}.

The positive definiteness of a T-Hermitian tensor can be determined by Lemma \ref{Lemma:BasicPropertiesTproduct}.(\romannumeral1). The following lemma about T-positive (semi-)definiteness is first observed over $\R$ as in \citep{zheng2021t}, and similar arguments also work for tensors over $\C$.

\begin{lemma}[\citep{zheng2021t}]\label{LemmaPositiveDefinitenessBlockDiagonal}
	Let $\mathcal{A} \in \mathbb{H}^{n \times n \times p}$. The followings are equivalent:
	\begin{itemize}
		\item [(\romannumeral1)] $\mathcal{A}$ is T-positive (semi-)definite;
		\item [(\romannumeral2)] $\bcirc(\mathcal{A})$ is Hermitian and positive (semi-)definite;
		\item [(\romannumeral3)] All the matrices $A_i$ are Hermitian and  positive (semi-)definite, where $A_1,...,A_p \in \C^{n \times n}$ are the diagonal blocks in (\ref{eqFourierBlockDiagMN}) (for $m=n$).
	\end{itemize}
\end{lemma}

Recall that a Hermitian positive (semi-)definite matrix has a unique positive (semi-)definite $k$th root. Hence, by Lemma  \ref{Lemma:BasicPropertiesTproduct}.(\romannumeral1), we obtain the following generalization of $k$th root of a T-positive (semi-)definite tensor.

\begin{lemma}[\citep{zheng2021t}]\label{LemmaKthRoot}
	For a positive integer $k$, a T-positive (semi-)definite $\mathcal{A} \in \mathbb{H}^{n \times n \times p}$ has a unique T-positive (semi-)definite $k$th root $\mathcal{B}\in \mathbb{H}^{n \times n \times p}$, that is, $\mathcal{B}^k=\mathcal{A}$. 
\end{lemma}

\bigskip
\subsection{T-eigenvalue decompositions}

We review the notion of \emph{T-eigenvalue, T-trace, and T-eigenvalue decomposition} in \citep{zheng2021t}. We will use the following definitions considering \citep[Proposition 4.1]{zheng2021t}.

\begin{definition}[\citep{zheng2021t}]\label{DefTEigen}
	Let $\mathcal{A} \in \C^{m \times n \times p}$. Then $\lambda \in \C$ is said to be a \emph{T-eigenvalue} of $\mathcal{A}$ if $\lambda$ is an eigenvalue of $\bcirc(\mathcal{A})$. We denote the multiset of all T-eigenvalues of $\mathcal{A}$ by $\spec(\mathcal{A})$. 
\end{definition}

\begin{definition}[\citep{zheng2021t}]\label{DefTTrace}
	Let $\mathcal{A}=\left[A^{(1)}\bpp \cdots\bpp A^{(p)}\right] \in \C^{n \times n \times p}$. Then the \emph{trace} of $\mathcal{A}$, denoted by $\tr(\mathcal{A})$, is defined as 
\begin{equation*}    
\tr(\mathcal{A})=\tr\left(\bcirc(\mathcal{A})\right)=p\sum_{i=1}^n(A^{(1)})_{ii}.
\end{equation*}
\end{definition}

The following proposition is straightforward by applying Lemma \ref{Lemma:BasicPropertiesTproduct}.(\romannumeral1).

\begin{proposition}[\citep{zheng2021t}]\label{PropTEigen}
	Let $\mathcal{A},\mathcal{B} \in \C^{n \times n \times p}$, and let $\mathcal{C} \in \C^{n \times n \times p}$ be invertible.
	\begin{itemize}
		\item [(\romannumeral1)] $\tr(\mathcal{A}*\mathcal{B})=\tr(\mathcal{B}*\mathcal{A})$.
		\item [(\romannumeral2)] $\tr(\mathcal{A})=\sum_{\lambda \in \spec(\mathcal{A})}\lambda$.
		\item [(\romannumeral3)] $\spec(\mathcal{C}^{-1}*\mathcal{A}*\mathcal{C})=\spec(\mathcal{A})$.
		\item [(\romannumeral4)] $\tr(\mathcal{C}^{-1}*\mathcal{A}*\mathcal{C})=\tr(\mathcal{A})$.
	\end{itemize}
\end{proposition}

In order to define T-eigenvalue decomposition, we need the notion of a unitary tensor.

\begin{definition}[\citep{kilmer2011factorization}]
	$\mathcal{Q} \in \C^{n \times n \times p}$ is \emph{unitary} if $\mathcal{Q}^H*\mathcal{Q}=\mathcal{Q}*\mathcal{Q}^H=\mathcal{I}_{n,p}$.
\end{definition}

Unitary tensors have a useful property about the Frobenius norm as similar as unitary matrices.

\begin{lemma}[\citep{kilmer2011factorization}]\label{LemmaUnitaryNorm}
	If $\mathcal{Q} \in \C^{n \times n \times p}$ is unitary, then $||\mathcal{Q}*\mathcal{A}||=||\mathcal{A}||$.
\end{lemma}

The existence of eigenvalue decomposition of a Hermitian matrix and Lemma \ref{LemmaFourierBlockDiag} ensure the existence of a T-eigenvalue decomposition of a T-Hermitian tensor as follows.

\begin{lemma}[\citep{zheng2021t}]
	Let $\mathcal{A} \in \mathbb{H}^{n \times n \times p}$. Then $\mathcal{A}$ can be factored as
	\begin{equation*}
		\mathcal{A}=\mathcal{U}*\mathcal{D}*\mathcal{U}^H,
	\end{equation*}
	where $\mathcal{U} \in \C^{n \times n \times p}$ is unitary and $\mathcal{D}=\left[\mathcal{D}^{(1)}\bpp \cdots\bpp \mathcal{D}^{(p)}\right] \in \C^{n \times n \times p}$  with diagonal matrices $D^{(i)}$ such that all the diagonal entries  of the block diagonalization 
	\begin{equation*}
		(\mathbf{F}_p \otimes I_n) \cdot \bcirc(\mathcal{D}) \cdot (\mathbf{F}_p^H \otimes I_n)
	\end{equation*} 
	are T-eigenvalues of $\mathcal{A}$. This factorization is called a \emph{T-eigenvalue decomposition} of $\mathcal{A}$.
\end{lemma}

We will use the following proposition in Section \ref{Sect:Riemannian Geometry}.

\begin{proposition}\label{PropFroEig}
	If $\mathcal{A} \in  \mathbb{H}^{n \times n \times p}$ has a T-eigenvalue decomposition $\mathcal{A}=\mathcal{U}*\mathcal{D}*\mathcal{U}^H$, then
	\begin{equation*}
		||\mathcal{D}||^2=||\mathcal{A}||^2=\frac{1}{p}\left( \sum_{\lambda \in \spec(\mathcal{A})}\lambda^2\right).
	\end{equation*} 
\end{proposition}
\begin{proof}
	The first equality holds by Lemma \ref{LemmaUnitaryNorm}.  Considering Definition \ref{DefTEigen}, the second equality holds since $||\mathcal{A}||^2=\frac{1}{p}||\bcirc(\mathcal{A})||^2$.
\end{proof}

 As one can easily notice in this section, the key observation is that Lemma \ref{Lemma:BasicPropertiesTproduct}.(\romannumeral1) converts problems of third-order tensor to problems of block circulant matrices.

\bigskip
\section{Geometric mean of two T-positive definite tensors}\label{Sect:Geometric Mean of T-positive Definite Tensors}

\subsection{Definition and basic properties}

The goal of this section is to generalize the geometric mean \eqref{eqMatrixGeometricMean} for T-positive definite tensors. To obtain a well-defined notion, we first check the following lemma in detail. 

\begin{lemma}\label{LemmaBinaryOper}
	Let $\mathcal{A},\mathcal{B} \in \mathbb{H}^{n \times n \times p}_{++}$.
	\begin{itemize}
		\item [(\romannumeral1)] $\mathcal{A}*\mathcal{B}*\mathcal{A} \in \mathbb{H}^{n \times n \times p}_{++}$.
		\item [(\romannumeral2)] $\mathcal{A}$ is invertible, and $\mathcal{A}^{-1} \in \mathbb{H}_{++}^{n \times n \times p}$.
		\item [(\romannumeral3)] $(\mathcal{A}^{\frac{1}{k}})^{-1}=(\mathcal{A}^{-1})^{\frac{1}{k}}$ for any positive integer $k$. We denote it by $\mathcal{A}^{-\frac{1}{k}}$.
  	\item [(\romannumeral4)] $\mathcal {A}^r \in \mathbb{H}_{++}^{n \times n \times p}$ for an arbitrary real number $r$.
	\end{itemize}
\end{lemma}
\begin{proof} We prove by dealing with the block circulant matricization of each object.
	\begin{itemize}
		\item [(\romannumeral1)] By Lemma \ref{LemmaPositiveDefinitenessBlockDiagonal}, it suffices to show that $\bcirc(\mathcal{A}*\mathcal{B}*\mathcal{A})$ is Hermitian positive definite. By Lemma \ref{Lemma:BasicPropertiesTproduct}.(\romannumeral1), $\bcirc(\mathcal{A}*\mathcal{B}*\mathcal{A})=\bcirc(\mathcal{A})\cdot \bcirc(\mathcal{B})\cdot\bcirc(\mathcal{A})$.
			Here, both $\bcirc(\mathcal{A})$ and $\bcirc(\mathcal{B})$ are Hermitian positive definite matrices by Lemma \ref{LemmaPositiveDefinitenessBlockDiagonal} again. Thus, the product $\bcirc(\mathcal{A})\cdot \bcirc(\mathcal{B})\cdot\bcirc(\mathcal{A})$ is Hermitian positive definite \citep[Observation 7.1.8]{horn2012matrix} as well.
		\item [(\romannumeral2)] By Lemma \ref{LemmaPositiveDefinitenessBlockDiagonal}, $\bcirc(\mathcal{A})$ is Hermitian positive definite and thus invertible. Hence, $\mathcal{A}$ is invertible by Lemma \ref{LemmaInvertibility}. Moreover, $\bcirc(\mathcal{A}^{-1})=\bcirc(\mathcal{A})^{-1}$ is also Hermitian positive definite \citep[p.430]{horn2012matrix}. We conclude that $\mathcal{A}^{-1} \in \mathbb{H}_{++}^{n \times n \times p}$ by applying Lemma \ref{LemmaPositiveDefinitenessBlockDiagonal} again. 
		\item [(\romannumeral3)] We have $\bcirc((\mathcal{A}^{\frac{1}{k}})^{-1})=(\bcirc(\mathcal{A}^{\frac{1}{k}}))^{-1}=(\bcirc(\mathcal{A})^{\frac{1}{k}})^{-1}=(\bcirc(\mathcal{A})^{-1})^{\frac{1}{k}}=(\bcirc(\mathcal{A}^{-1}))^{\frac{1}{k}}=\bcirc((\mathcal{A}^{-1})^{\frac{1}{k}})$.
        \item [(\romannumeral4)] Let $k$ be a positive integer, and we prove that $\mathcal{A}^k \in  \mathbb{H}_{++}^{n \times n \times p}$ at first. By Lemma \ref{LemmaPositiveDefinitenessBlockDiagonal}, it suffices to show that $\bcirc(\mathcal{A}^k)$ is a Hermitian positive definite matrix. By Lemma \ref{Lemma:BasicPropertiesTproduct}.(\romannumeral1), we have $\bcirc(\mathcal{A}^k)=\bcirc(\mathcal{A})^k$. Since $\bcirc(\mathcal{A})$ is Hermitian positive definite, then its power $\bcirc(\mathcal{A})^k$ is also Hermitian positive definite \citep[Corollary 7.2.2]{horn2012matrix}. Thus, $\mathcal{A}^k \in  \mathbb{H}_{++}^{n \times n \times p}$. Combining it with Lemma \ref{LemmaKthRoot} and (\romannumeral3), we have that  $\mathcal{A}^q \in \mathbb{H}_{++}^{n \times n \times p}$ for any  $q \in \Q$. 
        
        The real power of a Hermitian positive definite matrix is well-defined, and it does not break the block circulant shape. Thus, we can extend $\mathcal{A}^q~(q \in \Q)$ to real exponents $r \in \mathbb{R}$ by constructing a convergent sequence $\{q_i\}$ from rational approximations $q_i$ to $r$.
	\end{itemize}
\end{proof}

Now we are ready to define the geometric mean of two T-positive definite tensors. We use the same symbol $\#$ for the geometric mean for T-positive definite tensors by the abuse of notation. 
 
\begin{definition}
	For $\mathcal{A},\mathcal{B} \in \mathbb{H}_{++}^{n \times n \times p}$, the geometric mean $\mathcal{A} \# \mathcal{B}$ of $\mathcal{A}$ and $\mathcal{B}$ is defined as
	\begin{equation*}
		\mathcal{A} \# \mathcal{B}=\mathcal{A}^{\frac{1}{2}}*(\mathcal{A}^{-\frac{1}{2}}*\mathcal{B}*\mathcal{A}^{-\frac{1}{2}})^{\frac{1}{2}}*\mathcal{A}^{\frac{1}{2}}.
	\end{equation*}
\end{definition}

 For any $\mathcal{A},\mathcal{B} \in \mathbb{H}_{++}^{n \times n \times p}$, their geometric mean $\mathcal{A} \# \mathcal{B}$ is a well-defined T-positive definite tensor by Lemma \ref{LemmaBinaryOper}, that is, the geometric mean $\#$ gives a binary operator on $\mathbb{H}_{++}^{n \times n \times p}$. Furthermore, we have $\bcirc(\mathcal{A} \# \mathcal{B})=\bcirc(\mathcal{A})^{\frac{1}{2}}(\bcirc(\mathcal{A})^{-\frac{1}{2}}\bcirc(\mathcal{B})\bcirc(\mathcal{A})^{-\frac{1}{2}})^{\frac{1}{2}}\bcirc(\mathcal{A})^{\frac{1}{2}}=\bcirc(\mathcal{A})\# \bcirc(\mathcal{B})$ by Lemma \ref{Lemma:BasicPropertiesTproduct}.(\romannumeral1). We immediately have the following proposition.

\begin{proposition}\label{PropBcircGeometricMean}
Let $\mathcal{A},\mathcal{B} \in \mathbb{H}_{++}^{n \times n \times p}$. Then
	\begin{equation*}
			\mathcal{A}\# \mathcal{B}=\bcirc^{-1}(\bcirc(\mathcal{A})\#\bcirc(\mathcal{B})).
    \end{equation*}
\end{proposition}

\smallskip
Lawson and Lim \citep{lawson2001geometric} showed that the geometric mean for positive definite matrices satisfy the properties which ``mean" usually follows: the identity property (or idempotence) $A \# A=A$, the inversion property $(A\# B)^{-1}=A^{-1}\# B^{-1}$, the commutative property $A \# B=B \# A$, and the transformation property $(C^HAC)\#(C^HBC)=C^H(A\# B)C$, for $A,B \in \mathcal{P}_N$ and $C \in \mathcal{H}_N$. The geometric mean for T-positive definite tensors also satisfies such properties.

\begin{theorem}\label{MainThm1}
	Let $\mathcal{A},\mathcal{B} \in \mathbb{H}_{++}^{n \times n \times p}$.
	\begin{itemize}			
		\item [(\romannumeral1)] (The Identity Property)
			\begin{equation*}
				\mathcal{A}\#\mathcal{A}=\mathcal{A}.
			\end{equation*} 
		\item [(\romannumeral2)] (The Inversion Property)
			\begin{equation*}
				(\mathcal{A}\#\mathcal{B})^{-1}=\mathcal{A}^{-1} \# \mathcal{B}^{-1}.
			\end{equation*} 
		\item [(\romannumeral3)] (The Commutative Property)
			\begin{equation*}
				\mathcal{A}\#\mathcal{B}=\mathcal{B}\#\mathcal{A}.
			\end{equation*} 
		\item [(\romannumeral4)] (The Transformation Property) When $\Gamma_{\mathcal{C}}$ is defined by
  $\Gamma_{\mathcal{C}}(\mathcal{A})=\mathcal{C}^H * \mathcal{A} * \mathcal{C}$ for invertible $\mathcal{C} \in \C^{n \times n \times p}$,	
			\begin{equation*}
				\Gamma_{\mathcal{C}}(\mathcal{A}) \# \Gamma_{\mathcal{C}}(\mathcal{B})=\Gamma_{\mathcal{C}}(\mathcal{A} \# \mathcal{B}).
			\end{equation*}
	\end{itemize}
\end{theorem}
\begin{proof} 
The block circulant matricization gives each property for Hermitian positive definite matrices, which holds as in \citep{lawson2001geometric}. By Proposition \ref{PropBcircGeometricMean}, the properties for T-positive definite tensors are true.
\end{proof}

\smallskip

Lawson and Lim also showed in \citep{lawson2001geometric} that the geometric mean of positive definite matrices $A,B$ is the unique positive definite solution of the algebraic Riccati matrix equation $XA^{-1}X=B$. We observe that a similar statement holds for third-order tensors.

\begin{proposition}\label{PropRiccati}
    Let $\mathcal{A},\mathcal{B} \in \mathbb{H}_{++}^{n \times n \times p}$.  Then $\mathcal{A}\# \mathcal{B}$ is the unique T-positive definite solution of the Riccati tensor equation
			\begin{equation}\label{eqRiccati}
				\mathcal{X}*\mathcal{A}^{-1}*\mathcal{X}=\mathcal{B}.
			\end{equation}	
\end{proposition}
\begin{proof}
    It is straightforward that $\mathcal{A} \# \mathcal{B}$ is a T-positive definite solution of the equation (\ref{eqRiccati}). For proving the uniqueness, let $\mathcal X, \mathcal Y \in \mathbb{H}_{++}^{n \times n \times p}$ be two T-positive definite solutions so that $\mathcal{X}*\mathcal{A}^{-1}*\mathcal{X}=\mathcal{B}=\mathcal{Y}*\mathcal{A}^{-1}*\mathcal{Y}$. First note that 
		\begin{align*}
			(\mathcal{A}^{-\frac{1}{2}}*\mathcal{X}*\mathcal{A}^{-\frac{1}{2}})^2 &= (\mathcal{A}^{-\frac{1}{2}}*\mathcal{X}*\mathcal{A}^{-\frac{1}{2}})*(\mathcal{A}^{-\frac{1}{2}}*\mathcal{X}*\mathcal{A}^{-\frac{1}{2}})\\
			&=\mathcal{A}^{-\frac{1}{2}}*\mathcal{X}*\mathcal{A}^{-1}*\mathcal{X}*\mathcal{A}^{-\frac{1}{2}}\\
			&=\mathcal{A}^{-\frac{1}{2}}*\mathcal{Y}*\mathcal{A}^{-1}*\mathcal{Y}*\mathcal{A}^{-\frac{1}{2}}\\
			&=(\mathcal{A}^{-\frac{1}{2}}*\mathcal{Y}*\mathcal{A}^{-\frac{1}{2}})^2.
		\end{align*}
  			By Lemma \ref{LemmaBinaryOper}.(\romannumeral2)-(\romannumeral3), both $\mathcal{A}^{-\frac{1}{2}}*\mathcal{X}*\mathcal{A}^{-\frac{1}{2}},  \mathcal{A}^{-\frac{1}{2}}*\mathcal{Y}*\mathcal{A}^{-\frac{1}{2}} \in \mathbb{H}^{n \times n \times p}_{++}$ are T-positive definite. By Lemma \ref{LemmaBinaryOper}.(\romannumeral1), their squares $(\mathcal{A}^{-\frac{1}{2}}*\mathcal{X}*\mathcal{A}^{-\frac{1}{2}})^2=(\mathcal{A}^{-\frac{1}{2}}*\mathcal{Y}*\mathcal{A}^{-\frac{1}{2}})^2 \in \mathbb{H}^{n \times n \times p}_{++}$ coincide, and the resulting tensor is also T-positive definite. By the uniqueness of the $k$th root as in Lemma \ref{LemmaKthRoot}, we conclude that 
		\begin{equation*}
			\mathcal{A}^{-\frac{1}{2}}*\mathcal{X}*\mathcal{A}^{-\frac{1}{2}}=\mathcal{A}^{-\frac{1}{2}}*\mathcal{Y}*\mathcal{A}^{-\frac{1}{2}}
		\end{equation*}
		and thus $\mathcal{X}=\mathcal{Y}$.
\end{proof}

The above proposition gives a motivation of algebraic Riccati tensor equation which is of the form
	\begin{equation*}		\mathcal{X}*\mathcal{A}*\mathcal{X}+\mathcal{B}*\mathcal{X}+\mathcal{X}*\mathcal{C}+\mathcal{D}=\mathcal{O}, 
    \end{equation*}
	where $\mathcal{A},\mathcal{B},\mathcal{C},\mathcal{D}\in \C^{n \times n \times p}$ are coefficients and $\mathcal{X} \in \C^{n \times n \times p}$ is a variable.
	Such tensor equations have not been studied very much, however, this algebraic Riccati tensor equation is equivalent to the algebraic Riccati matrix equation with block circulant matrix coefficients and variable:
	\begin{equation*}
		\bcirc(\mathcal{X})\cdot \bcirc(\mathcal{A}) \cdot \bcirc(\mathcal{X})+\bcirc(\mathcal{B})\cdot\bcirc(\mathcal{X})+\bcirc(\mathcal{X})\cdot \bcirc(\mathcal{C})+\bcirc(\mathcal{D})=O,
	\end{equation*}
	where $O:=\bcirc(\mathcal{O})$ is a zero matrix. 
	This observation leads to the following interpretation of the geometric mean as a solution of the system of algebraic Riccati matrix equations. 
 
 \begin{theorem}\label{ThmGeometricmean}
    Let $\mathcal{A}=\left[A^{(1)}\bpp \cdots\bpp A^{(p)}\right],~\mathcal{B}=\left[B^{(1)}\bpp \cdots\bpp B^{(p)}\right] \in \mathbb{H}_{++}^{n \times n \times p}$.
      Then the geometric mean of $\mathcal{A}$ and $\mathcal{B}$ is explicitly expressed as 
    \begin{equation*}
        \mathcal{A} \# \mathcal{B}=\bcirc^{-1}\left((\mathbf{F}_p^H \otimes I_n)\cdot\diag(A_1 \# B_1,...,A_p\# B_p)\cdot(\mathbf{F}_p \otimes I_n)\right),
    \end{equation*}
where for each $i=1,...,p$,
    \begin{align*}
     A_i = \sum\limits_{k=1}^{p} \omega^{(i-1)(k-1)}A^{(k)}~~\text{and}~~B_i = \sum\limits_{k=1}^{p} \omega^{(i-1)(k-1)}B^{(k)}.
    \end{align*}
 \end{theorem}
 \begin{proof}
      By Proposition \ref{PropRiccati}, it suffices to show that $(\mathbf{F}_p^H \otimes I_n)\cdot\diag(A_1 \# B_1,...,A_p\# B_p)\cdot(\mathbf{F}_p \otimes I_n)$ is the uniquely determined Hermitian positive definite solution of the equation
	\begin{equation}\label{eqTargetRiccati}
		\bcirc(\mathcal{X}) \cdot \bcirc(\mathcal{A})^{-1}\cdot \bcirc(\mathcal{X})=\bcirc(\mathcal{B}).
	\end{equation}
	Since all block circulant matrices are simultaneously block diagonalizable by Lemma \ref{LemmaFourierBlockDiag}, we can solve \eqref{eqTargetRiccati} by solving several algebraic Riccati matrix equations via the block diagonalization as follows. 
 
    Let $X_1,...,X_p$ be the matrices complying with $\bcirc(\mathcal{X})=(\mathbf{F}_p^H \otimes I_m)\cdot \operatorname{diag}(X_1,...,X_p)\cdot (\mathbf{F}_p \otimes I_n)$, then the equation \eqref{eqTargetRiccati} can be rewritten as: 

	\begin{equation}\label{eqAssociatedMatrixEquation}
		\begin{bmatrix}
			X_1 & & &\\
			& X_2 & &\\
			& & \ddots &\\
			& &  & X_p
		\end{bmatrix}\begin{bmatrix}
			A_1^{-1} & & &\\
			& A_2^{-1} & &\\
			& & \ddots &\\
			& &  & A_p^{-1}
		\end{bmatrix}\begin{bmatrix}
			X_1 & & &\\
			& X_2 & &\\
			& & \ddots &\\
			& &  & X_p
		\end{bmatrix}=\begin{bmatrix}
			B_1 & & &\\
			& B_2 & &\\
			& & \ddots &\\
			& &  & B_p
		\end{bmatrix}.
	\end{equation}
	Note that the equation \eqref{eqAssociatedMatrixEquation} is equivalent to the following system of $p$ algebraic Riccati matrix equations:
	\begin{equation}\label{eqSystem}
		\begin{cases}
			X_1A_1^{-1}X_1=B_1,\\
			X_2A_2^{-1}X_2=B_2,\\
			~~~~~~~~~~\vdots\\
			X_pA_p^{-1}X_p=B_p.\\
		\end{cases}
	\end{equation}
	The blocks $A_i,B_i$ are Hermitian and positive definite for all $i=1,\ldots,p$ by Lemma \ref{LemmaPositiveDefinitenessBlockDiagonal}. Thus, 
    \begin{equation*}
        \left[
        \begin{array}{cccc}
            \hspace{-0.2cm}A_1 \# B_1 & & &\\
            & \hspace{-0.4cm}A_2 \# B_2 & &\\
            & & \hspace{-0.4cm}\ddots & \\
            & & & \hspace{-0.4cm}A_p \# B_p\hspace{-0.2cm}
        \end{array}
        \right]
    \end{equation*}
 is the uniquely determined Hermitian positive definite solution of \eqref{eqAssociatedMatrixEquation}, and so the assertion holds.
 \end{proof}

\smallskip
\begin{example}
    Consider the tensors $\mathcal{A}=\left[A^{(1)}\bpp A^{(2)}\right], ~\mathcal{B}=\left[B^{(1)}\bpp B^{(2)}\right]\in \mathbb{R}^{3\times3\times2}$ defined by
    \begin{align*}
        &A^{(1)} = 
        \begin{bmatrix}
            \hphantom{+}6 & \hphantom{+}1 & \hphantom{+}2\\
            \hphantom{+}1 & \hphantom{+}8 & \hphantom{+}3\\
            \hphantom{+}2 & \hphantom{+}3 &10
        \end{bmatrix}, \quad
        A^{(2)} = 
        \begin{bmatrix}
           \hphantom{+}4 & \hphantom{+}1 & \hphantom{+}2\\
          \hphantom{+}1 & \hphantom{+}6 & \hphantom{+}4\\
          \hphantom{+}2 & \hphantom{+}4 & \hphantom{+}2\\ 
        \end{bmatrix},\\
        &B^{(1)} = 
        \begin{bmatrix}
            \hphantom{+}8   & -3  &  -3\\
      -3   &  \hphantom{+}6  &   \hphantom{+}1\\
      -3   &  \hphantom{+}1  &   \hphantom{+}8
        \end{bmatrix}, \quad
        B^{(2)} = 
        \begin{bmatrix}
            -6   &  \hphantom{+}2  &   \hphantom{+}5\\
       \hphantom{+}2  &  -2  &  -3\\
       \hphantom{+}5  &  -3  &  -2
        \end{bmatrix}.
    \end{align*}
    Then $\mathcal{A},\mathcal{B}$ are Hermitian positive definite tensors and their geometric mean $\mathcal{X}=[X^{(1)}\bpp X^{(2)}] := \mathcal{A} \# \mathcal{B}$ is as follows:
    \begin{align*}
        & X^{(1)} \approx 
        \begin{bmatrix}
            \hphantom{+}4.5916 &  -0.6057  &  \hphantom{+}0.1536\\
      -0.6057  &  \hphantom{+}5.1580  &  \hphantom{+}0.4850\\
       \hphantom{+}0.1536  &  \hphantom{+}0.4850  &  \hphantom{+}7.4309
        \end{bmatrix}, \quad 
         X^{(2)} \approx 
        \begin{bmatrix}
           -0.4400  &  \hphantom{+}0.3644  &  \hphantom{+}2.2243\\
       \hphantom{+}0.3644  &  \hphantom{+}1.4536  &  \hphantom{+}0.0987\\
       \hphantom{+}2.2243  &  \hphantom{+}0.0987 &  -0.0154
        \end{bmatrix}.
    \end{align*}
\end{example}
The example was numerically calculated by MATLAB R2023b.

\bigskip
\subsection{T-L\"owner order}\label{SubSect:Lowner Order}

Recall that the L\"owner order on the set of all Hermitian matrices is defined as $A<B$ (resp. $A \leq B$) if $B-A$ is positive (resp. semi-)definite. We can naturally generalize the L\"owner order on $\mathbb{H}^{n \times n \times p}$, and call it T-L\"owner order.

\begin{definition}
	For $\mathcal{A}, \mathcal{B} \in \mathbb{H}^{n \times n \times p}$, the T-L\"owner order $<$ (resp. $\leq$) is defined by
	\begin{equation*}
		\mathcal{A} < \mathcal{B}~ (\text{resp. }\mathcal{A} \leq \mathcal{B})~~ \text{if}~~ \mathcal{B} -\mathcal{A} \in \mathbb{H}_{++}^{n \times n \times p} ~(\text{resp. }\mathcal{B} -\mathcal{A} \in \mathbb{H}_{+}^{n \times n \times p}).
	\end{equation*} 
\end{definition}

We will not distinguish the notation $<$ and $\leq$ for L\"owner and T-L\"owner orders unless there is ambiguousness.  
By Lemma \ref{LemmaPositiveDefinitenessBlockDiagonal}, the following property holds:

\begin{proposition}\label{PropBcircLownerInequal}
    Let $\mathcal{A},\mathcal{B} \in \mathbb{H}^{n \times n \times p}$. Then $\mathcal{A} \leq \mathcal{B}$ if and only if $ \bcirc(\mathcal{A}) \leq \bcirc(\mathcal{B})$.
\end{proposition}  

\begin{proposition}\label{PropLoewner}
	Let $\mathcal{A}, \mathcal{B},\mathcal{A}',\mathcal{B}' \in \mathbb{H}_{++}^{n \times n \times p}$ and $\mathcal{C} \in \mathbb{H}^{n \times n \times p}$.
	\begin{itemize}
		\item [(\romannumeral1)] If $\mathcal{A}<\mathcal{B}$, then $\mathcal{B}^{-1}<\mathcal{A}^{-1}$.
		\item [(\romannumeral2)] (The harmonic-geometric-arithmetic mean inequality)
			\begin{equation*}
				2(\mathcal{A}^{-1}+\mathcal{B}^{-1})^{-1} \leq \mathcal{A} \# \mathcal{B} \leq \frac{1}{2}(\mathcal{A}+\mathcal{B}).
			\end{equation*}
		\item [(\romannumeral3)] (The monotone property) If $\mathcal{A}' \leq \mathcal{A}$ and $\mathcal{B}' \leq \mathcal{B}$, then $\mathcal{A}' \# \mathcal{B}' \leq \mathcal{A} \# \mathcal{B}$.
		\item [(\romannumeral4)] (The L\"owner-Heinz inequality) If $\mathcal{C}^2 \leq \mathcal{A} \leq \mathcal{B}$, then $\mathcal{C} \leq \mathcal{A}^{\frac{1}{2}} \leq \mathcal{B}^{\frac{1}{2}}$.
	\end{itemize}
\end{proposition}
\begin{proof}
    The block circulant matricization gives each inequality for Hermitian positive definite matrices, which holds as in \citep{lawson2001geometric}. By Proposition \ref{PropBcircGeometricMean} and \ref{PropBcircLownerInequal}, the inequalities for T-positive definite tensors are true.
\end{proof}

\begin{proposition}\label{PropHermitianPD}
    Let $\mathcal{A}, \mathcal{B} \in \mathbb{H}_{++}^{n \times n \times p}$.
    \begin{itemize}
        \item [(\romannumeral1)] $\bcirc(\mathcal{A}\#\mathcal{B})$ is the largest matrix of all $np \times np$ Hermitian matrices $X$ for which the Hermitian and block matrix
			\begin{equation}\label{eqHermitianBlockPositive}
				\begin{bmatrix}
					\bcirc(\mathcal{A}) & X\\
					X & \bcirc(\mathcal{B})
				\end{bmatrix}
			\end{equation}
			is positive semi-definite.
		\item [(\romannumeral2)] For $\mathcal{X} \in \mathbb{H}^{n \times n \times p}$, $\mathcal{X}*\mathcal{A}^{-1}*\mathcal{X} \leq \mathcal{B}$ if and only if $\mathcal{X} \leq \mathcal{A}\# \mathcal{B}$.
    \end{itemize}
\end{proposition}
\begin{proof}
    By Lemma \ref{LemmaPositiveDefinitenessBlockDiagonal}, $\bcirc(\mathcal{A})$ and $\bcirc(\mathcal{B})$ are  Hermitian positive definite.
    \begin{itemize}
        \item [(\romannumeral1)] By \citep[Theorem 3.4]{lawson2001geometric}, the largest matrix of $n \times n$ Hermitian matrices $X$ for which \eqref{eqHermitianBlockPositive} is positive semi-definite is $\bcirc(\mathcal{A})\#\bcirc(\mathcal{B})$. In addition, we have $\bcirc(\mathcal{A})\#\bcirc(\mathcal{B})=\bcirc(\mathcal{A}\#\mathcal{B})$ by Proposition \ref{PropBcircGeometricMean}.
        \item [(\romannumeral2)] For $\mathcal{X} \in \mathbb{H}^{n \times n \times p} $, we have that $\bcirc(\mathcal{X})$ is Hermitian positive definite by Lemma \ref{LemmaPositiveDefinitenessBlockDiagonal}, and so
        \begin{align*}
            \mathcal{X}*\mathcal{A}^{-1}*\mathcal{X} \leq \mathcal{B} &\Leftrightarrow \bcirc(\mathcal{X})\cdot\bcirc(\mathcal{A})^{-1}\cdot\bcirc(\mathcal{X}) \leq \bcirc(\mathcal{B})~(\because \text{Proposition \ref{PropBcircLownerInequal}})\\
            &\Leftrightarrow \bcirc(\mathcal{X}) \leq \bcirc(\mathcal{A}) \# \bcirc(\mathcal{B})~(\because \text{\citep[Corollary 3.5]{lawson2001geometric}} ) \\
            &\Leftrightarrow \bcirc(\mathcal{X}) \leq \bcirc(\mathcal{A} \# \mathcal{B})~(\because \text{Proposition \ref{PropBcircGeometricMean}} ) \\
            &\Leftrightarrow \mathcal{X} \leq \mathcal{A}\# \mathcal{B}~(\because \text{Proposition \ref{PropBcircLownerInequal}} ).
        \end{align*}
    \end{itemize}
\end{proof}

\bigskip
\section{Riemannian geometry}\label{Sect:Riemannian Geometry}

Let ${\mathcal P}_N$ be the convex cone of $N \times N$ Hermitian positive definite matrices, and let $P$ be a point on ${\mathcal P}_N$. Then the tangent space $T_P\mathcal{P}_N$ of $\mathcal{P}_N$ at $P$ is the Euclidean space ${\mathcal H}_N$ of $N \times N$ Hermitian matrices. In \citep{bhatia2006riemannian, moakher2005differential}, a Riemannian manifold associated with the geometric mean for Hermitian positive definite matrices is introduced as the trace metric $\tr (P^{-1}XP^{-1}Y),$ for $X,Y \in T_P\mathcal{P}_N=\mathcal{H}_N$. Then the Riemannian distance $\delta$ with respect to the above metric is given by $\delta(A,B) = \| \log (A^{-1/2}BA^{-1/2}) \|_F$ for $A,B \in{\mathcal P}_N$. 
There are two important observations on this Riemannian manifold. First, the unique (up to parametrization) geodesic joining $A$ and $B$ is given by the curve of weighted geometric means defined as in \eqref{eqw-g-mean} so that the geometric mean $A \# B$ is the midpoint of this geodesic. Second, this Riemannian manifold is indeed a Cartan-Hadamard manifold. Similarly, we introduce a Riemannian manifold associated with the geometric mean for T-positive definite tensors and establish interesting results.

\bigskip
\subsection{Riemannian metric and geodesic}\label{SubSect:Riemannian Metric}

We give the topology induced by the Frobenius norm on $\mathbb{H}^{n \times n \times p}$. Then obviously $\mathbb{H}^{n \times n \times p}$ is a smooth manifold diffeomorphic to a Euclidean space. We have the following proposition which extends \citep{zheng2021t} over the base field $\C$.

\begin{proposition}\label{PropOpenConvexCone}
	$\mathbb{H}^{n \times n \times p}_{++}$ is a nonempty open convex cone on $\mathbb{H}^{n \times n \times p}$.
\end{proposition}

From Proposition \ref{PropOpenConvexCone}, we obtain that $\mathbb{H}^{n \times n \times p}_{++}$ is a smooth submanifold of $\mathbb{H}^{n \times n\times p}$. Now, we introduce a Riemannian metric on $\mathbb{H}^{n \times n \times p}_{++}$. Let $\mathcal{P} \in \mathbb{H}^{n \times n \times p}_{++}$. Since $\mathbb{H}^{n \times n \times p}_{++}$ is open in $\mathbb{H}^{n \times n \times p}$, the tangent space of $\mathbb{H}^{n \times n \times p}_{++}$ at $\mathcal{P}$ is nothing but 
\begin{equation*}
	T_{\mathcal{P}}\mathbb{H}^{n \times n \times p}_{++} =\mathbb{H}^{n \times n \times p}.
\end{equation*}
We define an inner product $g_{\mathcal{P}}( \cdot, \cdot )$ on each $T_{\mathcal{P}}\mathbb{H}^{n \times n \times p}_{++}$ by
\begin{equation}\label{eqNewInnerProduct}
	g_{\mathcal{P}}(\mathcal{X}, \mathcal{Y})=\tr(\mathcal{P}^{-1}*\mathcal{X}*\mathcal{P}^{-1}*\mathcal{Y})
\end{equation}
where $\mathcal{X},\mathcal{Y} \in T_{\mathcal{P}}\mathbb{H}^{n \times n \times p}_{++}$. 
These inner products defined throughout the smooth manifold $\mathbb{H}^{n \times n \times p}_{++}$ give us a Riemannian metric $g$ which is locally described by
\begin{equation}\label{eqds}
	ds=\sqrt{p}\cdot ||\mathcal{X}^{-\frac{1}{2}}*d\mathcal{X}*\mathcal{X}^{-\frac{1}{2}}||.
\end{equation} 
This is because
\begin{align*}
	ds^2&=\tr(\mathcal{X}^{-1}*d\mathcal{X}*\mathcal{X}^{-1}*d\mathcal{X})\\
	&=\tr(\mathcal{X}^{\frac{1}{2}}*\mathcal{X}^{-1}*d\mathcal{X}*\mathcal{X}^{-1}*d\mathcal{X}*\mathcal{X}^{-\frac{1}{2}})\\
	&=\tr(\mathcal{X}^{-\frac{1}{2}}*d\mathcal{X}*\mathcal{X}^{-\frac{1}{2}}*\mathcal{X}^{-\frac{1}{2}}*d\mathcal{X}*\mathcal{X}^{-\frac{1}{2}})\\
	&=\tr((\mathcal{X}^{-\frac{1}{2}}*d\mathcal{X}*\mathcal{X}^{-\frac{1}{2}})^2)\\
	&=(\text{sum of T-eigenvalues of}~(\mathcal{X}^{-\frac{1}{2}}*d\mathcal{X}*\mathcal{X}^{-\frac{1}{2}})^2)~(\because \text{Proposition \ref{PropTEigen}.(\romannumeral2)})\\
	&=(\text{sum of eigenvalues of}~(\bcirc(\mathcal{X}^{-\frac{1}{2}}*d\mathcal{X}*\mathcal{X}^{-\frac{1}{2}}))^2)~(\because \text{Definition \ref{DefTEigen}})\\
	&=(\text{sum of squares of eigenvalues of}~\bcirc(\mathcal{X}^{-\frac{1}{2}}*d\mathcal{X}*\mathcal{X}^{-\frac{1}{2}}))\\
	&=(\text{sum of squares of T-eigenvalues of}~\mathcal{X}^{-\frac{1}{2}}*d\mathcal{X}*\mathcal{X}^{-\frac{1}{2}})~(\because \text{Definition \ref{DefTEigen}})\\
	&=p\cdot ||\mathcal{X}^{-\frac{1}{2}}*d\mathcal{X}*\mathcal{X}^{-\frac{1}{2}}||^2~ (\because \text{Proposition \ref{PropFroEig}}).
\end{align*}

Now, we obtain a new Riemannian manifold $(\mathbb{H}^{n \times n \times p}_{++},g)$. By Definition \ref{DefTTrace}, the equation \eqref{eqNewInnerProduct} is equivalent to
\begin{equation}\label{eqIsometry}
	\begin{aligned}
	g_{\mathcal{P}}(\mathcal{X},\mathcal{Y})&=\tr\left(\bcirc(\mathcal{P})^{-1}\cdot \bcirc(\mathcal{X})\cdot \bcirc(\mathcal{P})^{-1}\cdot \bcirc(\mathcal{Y})\right)\\
	&=\hat{g}_{\bcirc(\mathcal{P})}(\bcirc(\mathcal{X}),\bcirc(\mathcal{Y})),
\end{aligned} 	
\end{equation}
where $\hat{g}$ denotes the Riemannian metric defined in \citep{bhatia2006riemannian}. We remark that \eqref{eqds} also can be achieved from \eqref{eqIsometry}. In addition, \eqref{eqIsometry} provides an important result in the aspect of the Riemannian submanifold of $\mathcal{P}_{np}$. Consider the block circulant matricizing map
\begin{equation*}
	\bcirc:(\mathbb{H}^{n\times n \times p}_{++}, g) \rightarrow (\mathcal{BCP}_{np}, ~\hat{g}|_{\mathcal{BCP}_{np}}) \subset (\mathcal{P}_{np}, \hat{g})
\end{equation*} 
defined by
\begin{equation}\label{eqbcircmap}
	\mathcal{P}=\left[P^{(1)}\Bpp P^{(2)}\Bpp \cdots \Bpp P^{(p)}\right] \mapsto \bcirc(\mathcal{P})=\begin{bmatrix}
			P^{(1)} & P^{(p)} & \cdots & P^{(2)}\\
			P^{(2)} & P^{(1)} &  \cdots & P^{(3)}\\
			\vdots & \vdots &  \ddots & \vdots\\
			P^{(p)} & P^{(p-1)}& \cdots & P^{(1)}
		\end{bmatrix},
\end{equation}	
where $\mathcal{BCP}_{np}$ denotes the set of block circulant $np \times np$ Hermitian positive definite  matrices. Considering \eqref{eqbcircmap}, the map $\bcirc$ is obviously smooth. Furthermore, it is an isometric embedding onto $\mathcal{BCP}_{np}$, by Lemma \ref{LemmaPositiveDefinitenessBlockDiagonal} and \eqref{eqIsometry}. Therefore, we have the following proposition.

\begin{proposition}
	$(\mathbb{H}^{n\times n \times p}_{++}, g)$ is an embedded submanifold of the Riemannian manifold $(\mathcal{P}_{np},\hat{g})$.
\end{proposition} 

Let $\mathcal{A},\mathcal{B} \in \mathbb{H}^{n\times n \times p}_{++}$ and consider $\bcirc(\mathcal{A}),\bcirc(\mathcal{B}) \in \mathcal{BCP}_{np}$. As \eqref{eqw-g-mean}, the uniquely determined geodesic $\hat{\gamma}:[0,1] \rightarrow \mathcal{P}_{np}$ from $\bcirc(\mathcal{A})$ to $\bcirc(\mathcal{B})$ is 
	\begin{equation}\label{eqGeodesicSubmanifold}
		\hat{\gamma}(t)=\bcirc(\mathcal{A})^{\frac{1}{2}}\cdot (\bcirc(\mathcal{A})^{-\frac{1}{2}}\cdot \bcirc(\mathcal{B}) \cdot \bcirc(\mathcal{A})^{-\frac{1}{2}})^t \cdot \bcirc(\mathcal{A})^{\frac{1}{2}}.
	\end{equation}
	Here, the important is that $\hat{\gamma}(t) \in \mathcal{BCP}_{np }$ for all $t \in [0,1]$. That is, $(\mathbb{H}^{n\times n \times p}_{++}, g)$ is path-connected. Moreover, as an embedded submanifold of $(\mathcal{P}_{np},\hat{g})$, it is totally geodesic, i.e., a geodesic on $(\mathbb{H}^{n\times n \times p}_{++}, g)$ is also a geodesic of the ambient space $(\mathcal{P}_{np},\hat{g})$. In addition, by Proposition \ref{PropOpenConvexCone}, $(\mathbb{H}^{n\times n \times p}_{++}, g)$ is simply connected. Consequently, we have the following theorem.

\begin{theorem}\label{MainThm2}
	$(\mathbb{H}^{n\times n \times p}_{++}, g)$ is a path-connected and simply connected totally geodesic submanifold of the Riemannian manifold $(\mathcal{P}_{np },\hat{g})$. 
\end{theorem}   

Since $\bcirc$ is an isometric embedding onto $\mathcal{BCP}_{np}$, we may consider an  isometric diffeomorphism
\begin{equation*}
	\bcirc^{-1}:(\mathcal{BCP}_{np},~\hat{g}|_{\mathcal{BCP}_{np}}) \rightarrow (\mathbb{H}^{n \times n \times p}_{++},g).
\end{equation*} 
For $\hat{\gamma}$ at \eqref{eqGeodesicSubmanifold}, $\gamma=\bcirc^{-1} \circ~\hat{\gamma}$ is the geodesic from $\mathcal{A}$ to $\mathcal{B}$ in $(\mathbb{H}^{n\times n \times p}_{++}, g)$, which is described (\ref{eqGeodesicTPD}) as below.

\begin{corollary}\label{MainCor1}
	For any $\mathcal{A},\mathcal{B} \in \mathbb{H}^{n \times n \times p}_{++}$, the geodesic $\gamma:[0,1] \rightarrow \mathbb{H}^{n \times n \times p}_{++}$ from $\mathcal{A}$ to $\mathcal{B}$ is uniquely determined (up to parametrization), and parametrized by
	\begin{equation}\label{eqGeodesicTPD}
		\gamma(t)=\mathcal{A}^{\frac{1}{2}}*(\mathcal{A}^{-\frac{1}{2}}*\mathcal{B}*\mathcal{A}^{-\frac{1}{2}})^t*\mathcal{A}^{\frac{1}{2}}.
	\end{equation}
\end{corollary}

Note that the above geodesic (\ref{eqGeodesicTPD}) is well-defined by Lemma \ref{LemmaBinaryOper}. Consequently, we immediately observe that the geometric mean of $\mathcal A, B$ indicates the midpoint of this unique geodesic, as same as the geometric interpretation of the geometric mean of two positive definite matrices.

\begin{corollary}\label{MainCor2}
	For any $\mathcal{A},\mathcal{B} \in \mathbb{H}^{n \times n \times p}_{++}$, $\mathcal{A} \# \mathcal{B}$ is the midpoint of the geodesic from $\mathcal{A}$ to $\mathcal{B}$.
\end{corollary}	

In the Riemannian manifold $(\mathbb{H}^{n \times n \times p}_{++},g)$, the equality (\ref{eqds}) implies that the length $L(\gamma)$ of a piecewise smooth path $\gamma:[a,b] \rightarrow \mathbb{H}^{n \times n \times p}_{++}$ is given by
\begin{equation}\label{eqLengthofPath}
	L(\gamma)=\sqrt{p} \cdot \int_a^b  ||\gamma^{-\frac{1}{2}}(t)*\gamma'(t)*\gamma^{-\frac{1}{2}}(t)||dt.
\end{equation}
Let $\delta(\mathcal{A},\mathcal{B})$ denote the distance between two points $\mathcal{A},\mathcal{B} \in (\mathbb{H}^{n \times n \times p}_{++},g)$, that is,
\begin{equation*}
	\delta(\mathcal{A},\mathcal{B})=\inf\{L(\gamma)~|~\gamma~\text{is a piecewise smooth path from $\mathcal{A}$ to $\mathcal{B}$} \}.
\end{equation*}
Now, the space $\mathbb{H}^{n \times n \times p}_{++}$ can be considered as a metric space $(\mathbb{H}^{n \times n \times p}_{++},\delta)$ as well as a Riemannian manifold $(\mathbb{H}^{n \times n \times p}_{++},g)$. We distinguish them by denoting $(\mathbb{H}^{n \times n \times p}_{++},\delta)$ or $(\mathbb{H}^{n \times n \times p}_{++},g)$. 

It is a fundamental but important result that, for each invertible $\mathcal{C} \in \C^{n \times n \times p}$, the transformation $\Gamma_\mathcal{C}$ defined in Theorem \ref{MainThm1} is an isometry with respect to the length.

\begin{proposition}
	Let $\mathcal{X} \in \C^{n \times n \times p}$ be  invertible, and let $\mathcal{A}, \mathcal{B} \in \mathbb{H}^{n \times n \times p}_{++}$. Then $L(\Gamma_{\mathcal{X}} \circ \gamma)=L(\gamma)$. In addition, $\delta(\Gamma_{\mathcal{X}}(\mathcal{A}),\Gamma_{\mathcal{X}}(\mathcal{B}))=\delta(\mathcal{A},\mathcal{B})$.
\end{proposition}
\begin{proof}
	If $L(\Gamma_{\mathcal{X}} \circ \gamma)=L(\gamma)$ holds, then $\delta(\Gamma_{\mathcal{X}}(\mathcal{A}),\Gamma_{\mathcal{X}}(\mathcal{B}))=\delta(\mathcal{A},\mathcal{B})$ obviously holds by the definition of the Riemannian distance. Next, we claim that $L(\Gamma_{\mathcal{X}} \circ \gamma)=L(\gamma)$. It suffices to show 
			\begin{equation*}
				||(\mathcal{X}^H*\gamma(t)* \mathcal{X})^{-\frac{1}{2}}*(\mathcal{X}^H*\gamma(t)* \mathcal{X})'*(\mathcal{X}^H*\gamma(t)* \mathcal{X})^{-\frac{1}{2}}||=||\gamma^{-\frac{1}{2}}(t)*\gamma'(t)*\gamma^{-\frac{1}{2}}(t)||.
			\end{equation*} 
			By Proposition \ref{PropFroEig}, we only need to check
			\begin{equation*}
				\spec((\mathcal{X}^H*\gamma(t)* \mathcal{X})^{-\frac{1}{2}}*(\mathcal{X}^H*\gamma(t)* \mathcal{X})'*(\mathcal{X}^H*\gamma(t)* \mathcal{X})^{-\frac{1}{2}})=\spec(\gamma^{-\frac{1}{2}}(t)*\gamma'(t)*\gamma^{-\frac{1}{2}}(t)).
			\end{equation*}
			By Proposition \ref{PropTEigen}.(\romannumeral3), we only need to compare T-eigenvalues of 
			\begin{equation*}
				(\mathcal{X}^H*\gamma(t)* \mathcal{X})'*(\mathcal{X}^H*\gamma(t)* \mathcal{X})^{-1}=	(\mathcal{X}^H*\gamma'(t)* \mathcal{X})*(\mathcal{X}^H*\gamma(t)* \mathcal{X})^{-1}~~\text{and}~~\gamma'(t)*\gamma^{-1}(t).
			\end{equation*}
			Note that $(\mathcal{X}^H*\gamma'(t)*\mathcal{X})*(\mathcal{X}^H*\gamma(t)*\mathcal{X})^{-1}=\mathcal{X}^H*(\gamma'(t)*\gamma^{-1}(t))*(\mathcal{X}^H)^{-1}$. Thus, by Proposition \ref{PropTEigen}.(\romannumeral3) again, $(\mathcal{X}^H*\gamma'(t)*\mathcal{X})*(\mathcal{X}^H*\gamma(t)*\mathcal{X})^{-1}$ and $\gamma'(t)*\gamma^{-1}(t)$ have the same T-eigenvalues. This completes the proof.
\end{proof}

For $(\mathbb{H}^{n\times n \times p}_{++}, g)$ being Cartan-Hadamard, it remains to show that this Riemannian manifold is complete and has nonpositive curvature. However, any totally geodesic submanifold of a complete Riemannian manifold with nonpositive curvature does not need to neither be complete nor have nonpositive curvature. 
We will follow two ways to provide a few more geometric description of $\mathbb{H}^{n\times n \times p}_{++}$. One is by using the analogue of the \emph{infinitesimal exponential metric increasing property} (IEMI for short) in order to mimic the arguments in \citep{bhatia2006riemannian}. Another is by considering the isometric embedding of $(\mathbb{H}^{n\times n \times p}_{++}, g)$ to $(\mathcal{BCP}_{np},~\hat{g}|_{\mathcal{BCP}_{np}})$ in order to directly make use of the results in \citep{bhatia2006riemannian}.

\bigskip
\subsection{Infinitesimal exponential metric increasing property}\label{SubSect:infinitesimal exponential metric increasing property}

In this section, we introduce \emph{exponential, logarithm and Fr\'echet derivative} for third-order tensors, in order to follow the arguments in \citep{bhatia2006riemannian}, in particular, to state and prove the IEMI for T-Hermitian tensors.

\begin{definition}
	For a frontal square tensor $\mathcal{A} \in \C^{n \times n \times p}$, the \emph{exponential} is defined as
	\begin{equation*}
		e^{\mathcal{A}}:=\exp(\mathcal{A})=\mathcal{I}_{n,p}+\sum_{k=1}^{\infty}\frac{1}{k!}\mathcal{A}^k=\mathcal{I}_{n,p}+\mathcal{A}+\frac{1}{2!}\mathcal{A}^2+\frac{1}{3!}\mathcal{A}^3+\cdots.
	\end{equation*}
\end{definition}

\begin{proposition}\label{PropExpWellDefined}
	The exponential map $\exp$ for $\mathcal{A} \in \C^{n \times n \times p}$ is well-defined. In particular,
	\begin{equation*}
		\bcirc(\exp(\mathcal{A}))=\exp(\bcirc(\mathcal{A})).
	\end{equation*}
\end{proposition}
\begin{proof}
	It suffices to show that $\bcirc(\exp(\mathcal{A}))$ is always well-defined. Note that
\begin{align*}
	\bcirc(\exp(\mathcal{A}))&=\bcirc(\mathcal{I}_{n,p})+\bcirc(\mathcal{A})+\frac{1}{2!}\bcirc(\mathcal{A}^2)+\frac{1}{3!}\bcirc(\mathcal{A}^3)+\cdots\\
	&=I_{np}+\bcirc(\mathcal{A})+\frac{1}{2!}\bcirc(\mathcal{A})^2+\frac{1}{3!}\bcirc(\mathcal{A})^3+\cdots\\
	&=\exp(\bcirc(\mathcal{A})),
\end{align*}
where the last $\exp$ denotes the usual exponential map for square matrices. Since powers and sum of block circulant matrices are again block circulant matrices, we conclude that  $\bcirc(\exp(\mathcal{A}))$ is also well-defined.
\end{proof}

Recall that the exponential map from the space of all $N \times N$ Hermitian matrices to the space of all $N \times N$ Hermitian positive definite matrices is a diffeomorphism. We can easily obtain a similar result.

\begin{proposition}\label{PropExpDiffeo}
	The exponential map $\exp:\mathbb{H}^{n \times n \times p} \rightarrow \mathbb{H}^{n \times n \times p}_{++}$ is a diffeomorphism.
\end{proposition}
\begin{proof}	
	Let $\mathcal{BCH}_{np}$ denote the space of block circulant Hermitian $np \times np$ matrices, and $\mathcal{BCP}_{np}$ the space of block circulant Hermitian positive definite $np \times np$ matrices.  By Proposition \ref{PropExpWellDefined}, for any $\mathcal{A} \in \mathbb{H}^{n \times n \times p}$, $\bcirc(\exp(\mathcal{A}))=\exp(\bcirc(\mathcal{A}))$, and so the following diagram commutes:
	\begin{equation*}
		\begin{matrix}
			\mathcal{BCH}_{np} & \xrightarrow{~~\exp~~} & \mathcal{BCP}_{np}\\
			~~~~~~\bigg\uparrow {\bcirc} & & ~~~~~~\bigg\uparrow {\bcirc}\\
			\mathbb{H}^{n \times n \times p} & \xrightarrow{~~\exp~~} & \mathbb{H}^{n \times n \times p}_{++}
		\end{matrix}
	\end{equation*}
	By Lemma \ref{LemmaHermitian} and Lemma \ref{LemmaPositiveDefinitenessBlockDiagonal}, the vertical maps are bijective. Since both maps $\bcirc$ and $\bcirc^{-1}$ are smooth, the vertical maps are indeed  diffeomorphisms. The upper horizontal map is a restriction of $\exp:\mathcal{H}_{np} \rightarrow \mathcal{P}_{np}$ to $\mathcal{BCH}_{np}$, and is surjective since the image block circulant matrix under $\exp$ is block circulant. Thus, the upper horizontal map is also a diffeomorphism.  Therefore, the assertion is true.
\end{proof}

Using Proposition \ref{PropExpDiffeo}, we may define the logarithmic map $\log$ on $\mathbb{H}^{n \times n \times p}_{++}$ as the inverse of the exponential map $\exp$ defined on $\mathbb{H}^{n \times n \times p}$.

\begin{definition}
	For $\mathcal{A} \in \mathbb{H}^{n \times n \times p}_{++}$, the \emph{logarithm} of $\mathcal{A}$, denoted by $\log(\mathcal{A})$, is defined as $\mathcal{B} \in \mathbb{H}^{n \times n \times p}$ such that $\exp(\mathcal{B})=\mathcal{A}$. That is, the logarithm on $\mathbb{H}^{n \times n \times p}_{++}$ is the inverse of $\exp:\mathbb{H}^{n \times n \times p} \rightarrow \mathbb{H}^{n \times n \times p}_{++}$.
\end{definition}

Now, we are going to consider the Fr\'echet derivative of $\exp:\mathbb{H}^{n \times n \times p} \rightarrow \mathbb{H}^{n \times n \times p}_{++}$. For normed vector spaces $(V,~||\cdot||_V),~(W,~||\cdot||_W)$ and an open subset $U$ of $V$, a continuous map $\phi:U \rightarrow W$ is said to be Fr\'echet differentiable at $x \in U$ if there exists a bounded linear operator $L:V \rightarrow W$ such that 
\begin{equation*}
	\lim_{h \rightarrow 0}\frac{||\phi(x+h)-\phi(x)-L(h)||_W}{||h||_V}=0.
\end{equation*}
If such $L$ exists, the $L$ is called the Fr\'echet derivative of $\phi$ at $x$ and denoted by $D\phi(x)$. Following the general definition, we define the Fr\'echet derivative of a continuous map from the normed space $\mathbb{H}^{n \times n \times p}=(\mathbb{H}^{n \times n \times p},||\cdot||_{\mathcal{F}})$ to it.

\begin{definition}
	A continuous map $\phi:\mathbb{H}^{n \times n \times p} \rightarrow \mathbb{H}^{n \times n \times p}$ is said to be \emph{Fr\'echet differentiable} at $\mathcal{X} \in \mathbb{H}^{n \times n \times p}$ if there exists a bounded linear operator $\mathcal{L}:\mathbb{H}^{n \times n \times p} \rightarrow \mathbb{H}^{n \times n \times p}$ such that
	\begin{equation*}
		\lim_{\mathcal{H} \rightarrow \mathcal{O}}\frac{||\phi(\mathcal{X}+\mathcal{H})-\phi(\mathcal{X})-\langle \mathcal{L},\mathcal{H}\rangle ||_\mathcal{F}}{||\mathcal{H}||_\mathcal{F}}=0.
	\end{equation*}
	If such $\mathcal{L}$ exists, then $\mathcal{L}$ is called the \emph{Fr\'echet derivative} of $\phi$ at $\mathcal{X}$ and denoted by $D\phi(\mathcal{X})$.
\end{definition}

Note that a continuous map $\phi:\mathbb{H}^{n \times n \times p} \rightarrow \mathbb{H}^{n \times n \times p}$ naturally induces a continuous map $\widetilde{\phi}$ from the set of unfoldings of T-Hermitian tensors to it and $\widehat{\phi}$ from the set of block circulant Hermitian matrices to it, respectively, defined by
\begin{equation*}
	\widetilde{\phi}(\unfold(\mathcal{A}))=\unfold(\phi(\mathcal{A}))
\end{equation*} 
and
\begin{equation*}
	\widehat{\phi}(\bcirc(\mathcal{A}))=\bcirc(\phi(\mathcal{A})).
\end{equation*}
If we use a similar notation for Fr\'echet derivative of matrix functions, the following property holds.

\begin{proposition}\label{PropFrechet}
	Let $\phi:\mathbb{H}^{n \times n \times p} \rightarrow \mathbb{H}^{n \times n \times p}$ be a continuous map. Then the followings are equivalent.
	\begin{itemize}
		\item [(\romannumeral1)] $\phi$ is Fr\'echet differentiable at $\mathcal{X} \in \mathbb{H}^{n \times n \times p}$.
		\item [(\romannumeral2)] $\widetilde{\phi}$ is Fr\'echet differentiable at $\unfold(\mathcal{X})$.
		\item [(\romannumeral3)] $\widehat{\phi}$ is Fr\'echet differentiable at $\bcirc(\mathcal{X})$.
	\end{itemize} 
	In particular, if $\phi$ is Fr\'echet differentiable at $\mathcal{X} \in \mathbb{H}^{n \times n \times p}$, then 
	\begin{equation*}
		D\phi(\mathcal{X})=\fold(D\widetilde{\phi}(\unfold(\mathcal{X})))=\bcirc^{-1}(D\widehat{\phi}(\bcirc(\mathcal{X}))).
	\end{equation*}
\end{proposition}
\begin{proof}
	Considering the definition of $\unfold$, we obviously obtain 
	\begin{align*}
		&\lim_{\mathcal{H} \rightarrow \mathcal{O}}\frac{||\phi(\mathcal{X}+\mathcal{H})-\phi(\mathcal{X})-\langle \mathcal{L},\mathcal{H}\rangle||_\mathcal{F}}{||\mathcal{H}||_\mathcal{F}}=0\\
		&\Leftrightarrow \lim_{\unfold(\mathcal{H}) \rightarrow \mathcal{O}}\frac{||\widetilde{\phi}(\unfold(\mathcal{X})+\unfold(\mathcal{H}))-\widetilde{\phi}(\unfold(\mathcal{X}))-\langle \unfold(\mathcal{L}),\unfold(\mathcal{H})\rangle||_F}{||\unfold(\mathcal{H})||_F}=0\\
		&\Leftrightarrow \lim_{\bcirc(\mathcal{H}) \rightarrow \mathcal{O}}\frac{||\widehat{\phi}(\bcirc(\mathcal{X})+\bcirc(\mathcal{H}))-\widehat{\phi}(\bcirc(\mathcal{X}))-\langle \bcirc(\mathcal{L}),\bcirc(\mathcal{H})\rangle||_F}{||\bcirc(\mathcal{H})||_F}=0.
	\end{align*}
	Thus, the proof is done.
\end{proof}

Now we prove an important lemma so called the IEMI for T-Hermitian tensors. For the time being, adopt the notation of the exponential function for Hermitian matrices as identical to that used for T-Hermitian tensors. Considering an orthonormal basis of eigenvectors for $H \in \mathcal{H}_N$,  the Fr\'{e}chet derivative of $e^H$ at $K \in \mathcal{H}_N$ is given by
\begin{equation*}
	De^H(K)=\left[ \frac{e^{\lambda_i}-e^{\lambda_j}}{\lambda_i-\lambda_j} \right] \circ K,
\end{equation*} 
where $\lambda_1,...,\lambda_N$ are the eigenvalues of $H$ and $\circ$ denotes the Hadamard product, i.e., entry-wise product \citep{bhatia2006riemannian}, and hence
\begin{align*}
	&(e^H)^{-\frac{1}{2}}De^H(K)(e^H)^{-\frac{1}{2}}\\
	&=\diag\big(e^{-\lambda_i/2}~|~i=1,...,N\big)\left( \left[ \frac{e^{\lambda_i}-e^{\lambda_j}}{\lambda_i-\lambda_j} \right] \circ K\right)\diag\big(e^{-\lambda_j/2}~|~j=1,...,N\big)\\
	&=\left[ \frac{e^{(\lambda_i-\lambda_j)/2}-e^{-(\lambda_i-\lambda_j)/2}}{\lambda_i-\lambda_j} \right] \circ K.
\end{align*}
The fact $\frac{e^{t/2}-e^{-t/2}}{t} \geq 1$ for all $t \in \R$ implies an inequality $||(e^H)^{-\frac{1}{2}}De^H(K)(e^H)^{-\frac{1}{2}}||_F \geq ||K||_F$ so called the IEMI for Hermitian matrices. We will prove the following IEMI for T-Hermitian tensors using this argument as in \citep{bhatia2006riemannian}. 

\begin{lemma}[IEMI for T-Hermitian tensors]\label{LemmaIEMI}
	Let $\mathcal{H},\mathcal{K} \in \mathbb{H}^{n \times n \times p}$. Then
	\begin{equation*}
		||(e^{\mathcal{H}})^{-\frac{1}{2}}*De^{\mathcal{H}}(\mathcal{K})*(e^{\mathcal{H}})^{-\frac{1}{2}}|| \geq ||\mathcal{K}||.
	\end{equation*}
\end{lemma}
\begin{proof}
	At first, we represent the left-hand side of the assertion using T-eigenvalues of $\mathcal{H}$. By Proposition \ref{PropExpWellDefined} and \ref{PropFrechet},
	\begin{equation}\label{eqbcircFrechet}
		\bcirc((e^{\mathcal{H}})^{-\frac{1}{2}}*De^{\mathcal{H}}(\mathcal{K})*(e^{\mathcal{H}})^{-\frac{1}{2}})=(e^{\bcirc(\mathcal{H})})^{-\frac{1}{2}}\cdot (De^{\bcirc(\mathcal{H})}(\bcirc(\mathcal{K})))\cdot (e^{\bcirc(\mathcal{H})})^{-\frac{1}{2}}.
	\end{equation}
	By the argument above, we obtain that the right-hand side of (\ref{eqbcircFrechet}) is
	\begin{equation*}
		\left[ \frac{e^{(\lambda_i-\lambda_j)/2}-e^{-(\lambda_i-\lambda_j)/2}}{\lambda_i-\lambda_j} \right] \circ \bcirc(\mathcal{K}).
	\end{equation*}
	where $\lambda_1,...,\lambda_{np}$ are the T-eigenvalues of $\mathcal{H}$ (which are real). Since $\frac{e^{t/2}-e^{-t/2}}{t} \geq 1$ for all $t \in \R$, then 
	\begin{align*}
		||(e^{\mathcal{H}})^{-\frac{1}{2}}*De^{\mathcal{H}}(\mathcal{K})*(e^{\mathcal{H}})^{-\frac{1}{2}}||^2 &=\frac{1}{p}||\bcirc((e^{\mathcal{H}})^{-\frac{1}{2}}*De^{\mathcal{H}}(\mathcal{K})*(e^{\mathcal{H}})^{-\frac{1}{2}})||^2 \\
		&=\frac{1}{p}\left|\left|\left[ \frac{e^{(\lambda_i-\lambda_j)/2}-e^{-(\lambda_i-\lambda_j)/2}}{\lambda_i-\lambda_j} \right] \circ \bcirc(\mathcal{K})\right|\right|^2\\
		&\geq \frac{1}{p}||\bcirc(\mathcal{K})||_F^2\\
		&\geq ||\mathcal{K}||^2.
	\end{align*}
\end{proof}

\bigskip
\subsection{Completeness and curvature}\label{SubSect:Completeness}

\begin{lemma}\label{LemmaLowerBoundofLength}
	Let $\gamma:[a,b] \rightarrow \mathbb{H}^{n \times n \times p}_{++}$ be a smooth path parametrized as $\gamma(t)=e^{\mathcal{H}(t)}$ (then $\mathcal{H}(t)=\log \gamma(t)$). Then
	\begin{equation*}
		L(\gamma) \geq 
\sqrt{p} \cdot \int_a^b||\mathcal{H}'(t)||dt.
	\end{equation*}
	In addition, for any $\mathcal{A},\mathcal{B} \in \mathbb{H}^{n \times n \times p}_{++}$,
	\begin{equation}\label{eqMainDistance1}
		\delta(\mathcal{A},\mathcal{B}) \geq \sqrt{p} \cdot ||\log(\mathcal{A})-\log(\mathcal{B})||.
	\end{equation}
\end{lemma}
\begin{proof}
	By the chain rule, $\gamma'(t)=De^{\mathcal{H}(t)}(\mathcal{H}'(t))$. By IEMI (Lemma \ref{LemmaIEMI}) with $\mathcal{K}=\mathcal{H}'(t)$, we obtain
	\begin{equation*}
		||(e^{\mathcal{H}(t)})^{-\frac{1}{2}}*De^{\mathcal{H}(t)}(\mathcal{H}'(t))*(e^{\mathcal{H}(t)})^{-\frac{1}{2}}|| \geq ||\mathcal{H}'(t)||.
	\end{equation*}
	By the formula (\ref{eqLengthofPath}) of $L(\gamma)$,
	\begin{equation*}
		L(\gamma)= \sqrt{p} \cdot \int_a^b||(e^{\mathcal{H}(t)})^{-\frac{1}{2}}*De^{\mathcal{H}(t)}(\mathcal{H}'(t))*(e^{\mathcal{H}(t)})^{-\frac{1}{2}}||dt \geq \sqrt{p} \cdot \int_a^b||\mathcal{H}'(t)||dt.
	\end{equation*}
	Now, we prove the second statement. Assume that $\gamma$ is from $\mathcal{A}$ to $\mathcal{B}$. Note that $\mathcal{H}(t)=\log \gamma(t)$ defines a smooth (by Proposition \ref{PropExpDiffeo}) path in the Euclidean space $\mathbb{H}^{n \times n \times p}$. Since the Euclidean length of $\mathcal{H}(t)~(a \leq t \leq b)$ is $\int_a^b||\mathcal{H}'(t)||dt$, we have the following: 
	\begin{equation*}
		L(\gamma) \geq \sqrt{p} \cdot \int_a^b||\mathcal{H}'(t)||dt \geq \sqrt{p} \cdot ||\mathcal{H}(a)-\mathcal{H}(b)||= \sqrt{p} \cdot ||\log(\mathcal{A})-\log(\mathcal{B})||.
	\end{equation*}
	Thus, by the definition of $\delta(\mathcal{A},\mathcal{B})$, the second assertion holds.
\end{proof}

\begin{lemma}\label{LemmaExpIso}
	Let $\mathcal{A},\mathcal{B} \in \mathbb{H}^{n \times n \times p}_{++}$ be commuting T-positive definite tensors, that is, $\mathcal{A}*\mathcal{B}=\mathcal{B}*\mathcal{A}$. Then
	\begin{equation}\label{eqMainDistance2}
		\delta(\mathcal{A},\mathcal{B})=\sqrt{p} \cdot ||\log(\mathcal{A})-\log(\mathcal{B})||.
	\end{equation}
\end{lemma}
\begin{proof}
	At first, we prove that there exists a smooth path $\gamma$ with length $\sqrt{p} \cdot ||\log(\mathcal{A})-\log(\mathcal{B})||$. Consider the path $\gamma:[0,1] \rightarrow \mathbb{H}^{n \times n \times p}_{++}$ defined by
	\begin{equation*}
		\gamma(t)=\exp((1-t)\log(\mathcal{A})+t\log(\mathcal{B}))
	\end{equation*}
	which is from $\mathcal{A}$ to $\mathcal{B}$. Since $\mathcal{A}*\mathcal{B}=\mathcal{B}*\mathcal{A}$, we have $\gamma(t)=\mathcal{A}^{1-t}*\mathcal{B}^t$ and $\gamma'(t)=(\log(\mathcal{B})-\log(\mathcal{A}))\gamma(t)$. Hence, (\ref{eqLengthofPath}) implies 
	\begin{equation*}
		L(\gamma)=\sqrt{p} \cdot \int_0^1||\log(\mathcal{A})-\log(\mathcal{B})||dt=\sqrt{p} \cdot ||\log(\mathcal{A})-\log(\mathcal{B})||.
	\end{equation*}
	
	We claim that $\gamma$ is the unique (up to reparametrization) piecewise smooth path with length $\sqrt{p} \cdot ||\log(\mathcal{A})-\log(\mathcal{B})||$. Let $\widetilde{\gamma}$ be a piecewise smooth path from $\mathcal{A}$ to $\mathcal{B}$ of length $\sqrt{p} \cdot ||\log(\mathcal{A})-\log(\mathcal{B})||$. Then $\widetilde{\mathcal{H}}(t):=\sqrt{p} \cdot \log(\widetilde{\gamma}(t))$ is a smooth path from $\sqrt{p} \cdot \log(\mathcal{A})$ to $\sqrt{p} \cdot \log(\mathcal{B})$, which has Euclidean length $\sqrt{p} \cdot ||\log(\mathcal{A})-\log(\mathcal{B})||$ (see the proof of Lemma \ref{LemmaLowerBoundofLength}). In Euclidean space, such a path must be $H(t)=\sqrt{p} \cdot ((1-t)\log(\mathcal{A})+t\log(\mathcal{B}))$ (line segment). Throughout $[0,1]$, $H$ maps isometrically to the path $\gamma$. Thus, the uniqueness is proved.
\end{proof}

\begin{proposition}\label{PropMetricFormula}
	For any $\mathcal{A},\mathcal{B} \in \mathbb{H}^{n \times n \times p}_{++}$, 
 \begin{equation}\label{eqMainDistance3}
     \delta(\mathcal{A},\mathcal{B})=\sqrt{p} \cdot ||\log(\mathcal{A}^{-\frac{1}{2}}*\mathcal{B}*\mathcal{A}^{-\frac{1}{2}})||.
 \end{equation}
\end{proposition}
\begin{proof}
	Note that $\mathcal{I}_{n,p}$ and $\mathcal{A}^{-\frac{1}{2}}*\mathcal{B}*\mathcal{A}^{-\frac{1}{2}}$ commute. Hence, according to Corollary \ref{MainCor1}, the smooth path $\gamma_0:[0,1] \rightarrow \mathbb{H}^{n \times n \times p}_{++}$ defined by $\gamma_0(t)=(\mathcal{A}^{-\frac{1}{2}}*\mathcal{B}*\mathcal{A}^{-\frac{1}{2}})^t$ is the unique geodesic from $\mathcal{I}_{n,p}$ to $\mathcal{A}^{-\frac{1}{2}}*\mathcal{B}*\mathcal{A}^{-\frac{1}{2}}$. By applying the isometry (with respect to length) $\Gamma_{\mathcal{A}^{\frac{1}{2}}}$ to $\gamma_0$, we obtain the desired $\gamma$ since
	\begin{equation*}
		\gamma(t)=\Gamma_{\mathcal{A}^{\frac{1}{2}}}(\gamma_0(t))=\mathcal{A}^{\frac{1}{2}}*(\mathcal{A}^{-\frac{1}{2}}*\mathcal{B}*\mathcal{A}^{-\frac{1}{2}})^t*\mathcal{A}^{\frac{1}{2}}.
	\end{equation*}
	Since $\Gamma_{\mathcal{A}^{\frac{1}{2}}}$ is an isometry, $\gamma$ is a geodesic from $\mathcal{A}$ to $\mathcal{B}$. If there exists any other geodesic from $\mathcal{A}$ to $\mathcal{B}$, then the isometry $\Gamma_{\mathcal{A}^{-\frac{1}{2}}}$ induces another geodesic from $\mathcal{I}_{n,p}$ to $\mathcal{A}^{-\frac{1}{2}}*\mathcal{B}*\mathcal{A}^{-\frac{1}{2}}$ which is different with (i.e., not a reparametrization of) $\gamma_0$, that contradicts to the uniqueness of the geodesic. Hence, the geodesic from $\mathcal{A}$ to $\mathcal{B}$ is uniquely determined. Moreover, according to Lemma \ref{LemmaExpIso},
	\begin{align*}
		\delta(\mathcal{A},\mathcal{B})&=\delta(\mathcal{I}_{n,p},\mathcal{A}^{-\frac{1}{2}}*\mathcal{B}*\mathcal{A}^{-\frac{1}{2}})\\
		&=\sqrt{p} \cdot ||\log(\mathcal{I}_{n,p})-\log(\mathcal{A}^{-\frac{1}{2}}*\mathcal{B}*\mathcal{A}^{-\frac{1}{2}})||\\
		&=\sqrt{p} \cdot ||\log(\mathcal{A}^{-\frac{1}{2}}*\mathcal{B}*\mathcal{A}^{-\frac{1}{2}})||.
	\end{align*}
\end{proof}

In fact, we could prove the results (\ref{eqMainDistance1})--(\ref{eqMainDistance3}) without IEMI for T-positive definite tensors, as we mentioned at the last of Section \ref{SubSect:Riemannian Metric}.  Let $\hat{\delta}$ denote the distance associated to $(\mathcal{P}_{np},\hat{g})$. For a moment, we use the notation $\log$ for not only for T-positive definite tensors but also for Hermitian positive definite matrices.  For $A,B \in \mathcal{P}_N$, it was shown in \citep{bhatia2006riemannian} that $\hat{\delta}(A,B) \geq ||\log(A)-\log(B)||_F$; $\hat{\delta}(A,B) = ||\log(A)-\log(B)||_F$ if $A,B$ commute; and $\hat{\delta}(A,B) =||\log(A^{-\frac{1}{2}}BA^{-\frac{1}{2}})||_F$. Recall that $\bcirc:(\mathbb{H}^{n\times n \times  p}_{++},g,\delta) \rightarrow (\mathcal{BCP}_{np},\hat{g}|_{\mathcal{BCP}_{np}},\hat{\delta}|_{\mathcal{BCP}_{np}})$  is an isometry. In addition, considering the diagram on the proof of Proposition \ref{PropExpDiffeo}, we have that the logarithm and $\bcirc$ commute. Therefore, for $\mathcal{A},\mathcal{B} \in \mathbb{H}^{n\times n \times  p}_{++}$, 
\begin{equation*}
    \delta(\mathcal{A},\mathcal{B})=\hat{\delta}(\bcirc(\mathcal{A}),\bcirc(\mathcal{B}))
    \geq ||\log(\bcirc(\mathcal{A}))-\log(\bcirc(\mathcal{B}))||_F
    =||\bcirc(\log(\mathcal{A}))-\bcirc(\log(\mathcal{B}))||_F
\end{equation*}
which exactly implies (\ref{eqMainDistance1}), and similarly we can obtain (\ref{eqMainDistance2})-(\ref{eqMainDistance3}) either.

\begin{theorem}\label{ThmComplete}
	$(\mathbb{H}^{n \times n \times p}_{++},\delta)$ is complete.
\end{theorem}
\begin{proof}
	Let $\{\mathcal{A}_n\}$ be a Cauchy sequence in $(\mathbb{H}^{n \times n \times p}_{++},\delta)$. Since $\exp:\mathbb{H}^{n \times n \times p} \rightarrow \mathbb{H}^{n \times n \times p}_{++}$ is bijective by Proposition \ref{PropExpDiffeo}, for each positive integer $n$, there exists $\mathcal{X}_n \in \mathbb{H}^{n \times n \times p}$ such that $\mathcal{A}_n=\exp(\mathcal{X}_n)$. By Lemma \ref{LemmaLowerBoundofLength}, for any positive integers $m,n$,
	\begin{equation*}
		 ||\mathcal{X}_n-\mathcal{X}_m|| \leq \frac{1}{\sqrt{p}} \cdot\delta(\mathcal{A}_n,\mathcal{A}_m).
	\end{equation*} 
	Hence, $\{X_n\}$ is a Cauchy sequence in the Euclidean space $(\mathbb{H}^{n \times n \times p},||\cdot||)$. Since $(\mathbb{H}^{n \times n \times p},||\cdot||)$ is complete, $\{\mathcal{X}_n\}$ has a limit $\mathcal{X} \in \mathbb{H}^{n \times n \times p}$, that is,
	\begin{equation*}
		||\mathcal{X}_n-\mathcal{X}||=||\log(\mathcal{A}_n)-\mathcal{X}|| \rightarrow 0~~\text{as}~~n \rightarrow \infty.
	\end{equation*}
	Now, we  prove that $\exp:(\mathbb{H}^{n \times n \times p},||\cdot||) \rightarrow (\mathbb{H}^{n \times n \times p}_{++},\delta)$ is continuous, then clearly 
	\begin{equation*}
		\delta(\mathcal{A}_n,\exp(\mathcal{X})) \rightarrow 0 ~~\text{as}~~n \rightarrow \infty.
	\end{equation*}
	Suppose that $\mathcal{Y}_n \rightarrow \mathcal{Y}$ as $n \rightarrow \infty$ in $(\mathbb{H}^{n \times n \times p},||\cdot||)$. By Proposition \ref{PropMetricFormula},
	\begin{equation*}
		\delta(\exp(\mathcal{Y}_n),\exp(\mathcal{Y}))=\sqrt{p} \cdot ||\log(\exp(\mathcal{Y})^{-\frac{1}{2}}*\exp(\mathcal{Y}_n)*\exp(\mathcal{Y})^{-\frac{1}{2}})||.
	\end{equation*}
	In order to check $\delta(\exp(\mathcal{Y}_n),\exp(\mathcal{Y})) \rightarrow 0$, it suffices to observe that the T-eigenvalues of $\exp(\mathcal{Y})^{-\frac{1}{2}}*\exp(\mathcal{Y}_n)*\exp(\mathcal{Y})^{-\frac{1}{2}}$ tend to $1$ so that $\exp(\mathcal{Y})^{-\frac{1}{2}}*\exp(\mathcal{Y}_n)*\exp(\mathcal{Y})^{-\frac{1}{2}} \rightarrow \mathcal{I}_{n,p}$ as $n \rightarrow \infty$. It follows from that the eigenvalues of the block circulant matrix
	\begin{equation*}
		\bcirc(\exp(\mathcal{Y})^{-\frac{1}{2}}*\exp(\mathcal{Y}_n)*\exp(\mathcal{Y})^{-\frac{1}{2}})=\exp(\bcirc(\mathcal{Y}))^{-\frac{1}{2}}*\exp(\bcirc(\mathcal{Y}_n))*\exp(\bcirc(\mathcal{Y}))^{-\frac{1}{2}} 
	\end{equation*} 
	tend to $1$.
\end{proof}

Since $(\mathbb{H}^{n \times n \times p}_{++},g)$ is connected and its associated metric space $(\mathbb{H}^{n \times n \times p}_{++},\delta)$ is complete, the famous Hopf-Rinow theorem \citep[p.146-147]{do1992riemannian} implies the geodesic completeness of $(\mathbb{H}^{n \times n \times p}_{++},g)$. 

\begin{corollary}
	$(\mathbb{H}^{n \times n \times p}_{++},g)$ is (geodesically) complete.
\end{corollary}

In \citep{bhatia2006riemannian}, the author also proved that $(\mathcal{P}_{np},\hat{\delta})$ has nonpositive curvature, i.e., is a CAT(0) space \citep{bridson2013metric} by verifying that it satisfies the \emph{semi-parallelogram law} \citep[Proposition 5]{bhatia2006riemannian}. Since $(\mathbb{H}^{n\times n \times p}_{++}, g)$ is a totally geodesic submanifold of $(\mathcal{P}_{np \times np},\hat{g})$, then the metric space $(\mathbb{H}^{n\times n \times p}_{++}, \delta)$ must also satisfy the semi-parallelogram law. To sum up, we have the following analogous results for $\mathbb{H}^{n \times n \times p}_{++}$.

\begin{corollary}\label{MainCor3}
	$(\mathbb{H}^{n\times n \times p}_{++}, \delta)$ has nonpositive curvature, i.e., is a CAT(0) space.
\end{corollary}

Recall that for a complete Riemannian manifold, it has nonpositive curvature if and only if its metric space is a CAT(0) space. Thus, we have the following corollaries.

\begin{corollary}
	$(\mathbb{H}^{n\times n \times p}_{++}, g)$ has nonpositive curvature.
\end{corollary}

\begin{corollary}\label{CorCartanHadamard}
	$(\mathbb{H}^{n \times n \times p}_{++},g)$  is a Cartan-Hadamard-Riemannian manifold.
\end{corollary}

A Cartan-Hadamard-Riemannian manifold has a strong topological property that it is diffeomorphic to a Euclidean space \citep[p.149]{do1992riemannian}, thus so is $\mathbb{H}^{n \times n \times p}_{++}$.

\bigskip
\section{Final remarks}

The geometric mean of two positive definite matrices in \eqref{eqMatrixGeometricMean} is the metric midpoint of \( A \) and \( B \) for the trace metric on the set \( \mathcal{P} \) of positive definite matrices of some fixed dimension (see, e.g., \citep{bhatia2009positive}).
It is natural to consider an averaging technique over this metric to extend this mean to more than two positive definite matrices. Moakher \citep{moakher2005differential} and then Bhatia-Holbrook \citep{bhatia2006riemannian} suggested extending the geometric mean to \( n \) number of positive definite matrices $A_1, \ldots, A_n$ by taking the mean to be the unique minimizer of the sum of the squares of the distances:
\[ 
g_n(A_1, \ldots, A_n) = \underset{X > 0}{\mathrm{arg min}} \sum_{i=1}^{n} \delta^2(X, A_i),
\]
where \( \delta(X, A_i) = \|\log X^{-1/2}A_iX^{-1/2}\|_F \). This idea had been anticipated by Cartan and then Riemannian centers of mass in the setting of Riemannian manifolds was carried out by Karcher (see \citep{karcher1977riemannian} for more details).
Another approach to generalizing the geometric mean to \( n \)-variables, independent of metric notions, was suggested by Ando, Li, and Mathias \citep{ando2004geometric}. It established ten desirable properties for extended geometric means. 
Finding the geometric mean of T-positive definite tensors $\mathcal{A}_1,...,\mathcal{A}_n \in \mathbb{H}^{n \times n \times p}_{++}$ is clearly of interest for further work.

On the other hand, for $\mathcal{T}$ to be T-positive definite, $\bcirc(\mathcal{T})$ must be Hermitian positive definite with a special structure. For example, if the tensor $\mathcal{T}=\left[T^{(1)}\bpp T^{(2)}\bpp T^{(3)}\right]$ is a Hermitian T-positive definite, then $T^{(3)} = (T^{(2)})^{H}$, which is a rather special form of tensor. It is questionable how to relax this constraint to extend to a broader range of tensor structures. This will be left for future work.

\bigskip

\section*{Acknowledgments} 
The work of H.~Choi and T.~Kim
was supported by the National Research Foundation of Korea(NRF) grant funded by the Korea government(MSIT) (No. 2022R1A5A1033624). J.-H.~Ju and Y.~Kim were supported by the Basic Science Program of the NRF of Korea (NRF-2022R1C1C1010052). 

%\section*{References}
\bibliographystyle{abbrvnat}
\bibliography{GM}

\end{document}